\RequirePackage{amsmath}
\RequirePackage{fix-cm}
\documentclass[smallextended]{svjour3}       
\smartqed  
\usepackage{graphicx}
\usepackage{natbib}
\usepackage[caption=false]{subfig}
\usepackage[english]{babel}
\usepackage[utf8]{inputenc}
\usepackage[T1]{fontenc}
\usepackage{amsfonts}
\usepackage{hyperref}
\usepackage{dsfont}
\usepackage{variations}
\usepackage{multicol}
\usepackage[table]{xcolor}
\usepackage{dsfont}
\usepackage[squaren]{SIunits}
\usepackage{multirow}
\usepackage[toc,page]{appendix} 
\usepackage{ulem}
\newcommand{\itemb}{\item[$\bullet$]}

\newcommand{\dx}{\, {\mathrm d}}
\newcommand{\Dx}{\, {\mathrm d}}

\newcommand{\E}{\mathbb{E}}
\newcommand{\R}{\mathbb{R}}
\newcommand{\N}{\mathbb{N}}

\renewcommand{\P}{\mathbb{P}}

\newtheorem{hypo}{Hypothesis}

\definecolor{bl}{rgb}{0.0,0.2,0.6}

\begin{document}

\title{A stochastic model of cell adhesion to the vascular wall
}
%



\author{Christ\`{e}le Etchegaray         \and
        Nicolas Meunier 
}


\institute{C. Etchegaray \at
              INRIA Monc, Institut de Math\'ematiques de Bordeaux, 351, cours de la Libération,
 33 405 TALENCE, France.\\
              \email{christele.etchegaray@inria.fr}           
           \and
           N. Meunier \at
              LaMME, CNRS UMR 8071, Universit\'{e} \'Evry Val d'Essonne, 23 boulevard de France
91 037 \'{E}vry Cedex,
France. \\
\email{nicolas.meunier@univ-evry.fr} 
}

\date{Received: date / Accepted: date}

\maketitle

\begin{abstract}

This paper deals with the adhesive interaction arising between a cell circulating in the blood flow and the vascular wall. The purpose of this work is to investigate the effect of the blood flow velocity on the cell dynamics, and in particular on its possible adhesion to the vascular wall. 
We formulate a model that takes into account the stochastic variability of the formation of bonds, and the influence of the cell velocity on the binding dynamics: the faster the cell goes, the more likely existing bonds are to disassemble. The model is based on a nonlinear birth-and-death-like dynamics, in the spirit of \cite{joffe_weak_1986,ethier2009markov}.
We prove that, under different scaling regimes, the cell velocity follows either an ordinary differential equation or a stochastic differential equation, that we both analyse. We obtain both the identification of a shear-velocity threshold associated with the transition from cell sliding and its firm adhesion, and the expression of the cell mean stopping time as a function of its adhesive dynamics.

\end{abstract}

\indent \begin{tabular}{rlr}
1.& Introduction  & (\pageref{intro}) \\[0.cm]
2. & A Markovian jump process for the cell adhesion dynamics & (\pageref{sto})\\[0.cm]
3. & Continuous limiting models and characterization of the dynamics & (\pageref{cont})\\[0.cm]
4. & Discussion & (\pageref{discussion})\\[0.cm]
 & References &  (\pageref{ref})\\[0.cm]
 & Appendices & (\pageref{annex1})\\
\end{tabular}

\hrulefill 

\section{Introduction}\label{intro}
\subsection{Biological context}
Cell adhesion to the vascular wall is a major process involved e.g in inflammation and metastasis invasion \citep{Granger}. The adhesive interaction between a cell circulating in the blood flow and endothelial cells forming the wall occurs in the presence of hemodynamic forces.
 Adhesion bonds can form between cell transmembrane proteins, called \textit{ligands}, and adhesion receptors at the vascular wall surface. The first step of interaction happens when some bonds form and are stabilized, therefore slowing the cell down: this is the so called \textbf{capture phase}. Then, the cell rolls or slides along the endothelium, as new bonds form in the direction of motion and bonds at the back disassemble. This step is mediated by the selectin receptor molecules. During \textbf{rolling}, endothelial cells may also be stimulated, leading to the recruitment of integrins, another family of cell-receptor molecules. They mediate the cell \textbf{firm adhesion} which allows it to cross the vascular. The cell \textbf{extravasation} out of the vessel leads to further development of the process involved (immune response, invasion of tissues by metastatic cells e.g, see for example \cite{ley_getting_2007}). This process is depicted in Figure \ref{fig:leuco} for leukocytes.
 
The fate of a rolling cell is not clear, since it can either adhere to the wall or be released in the blood flow. Many experiments have been dedicated to the study, in vivo and in vitro, of isolated cells rolling on either monolayers of cultured endothelial cells or surfaces coated with selectin or other molecules. This has helped putting to light the adhesive interaction between rolling cells and the substrate \citep{Springer}, and also the stochastic variability of the rolling motion \citep{Schmid-Schonbein,Goetz}. In particular, for cells rolling on a substrate bearing a uniform concentration of adhesion molecules \citep{Springer}, such a variability suggests that the fluctuations in the cell dynamics near the vascular wall reflects the stochastic nature of the binding dynamics. Our purpose is to both build and study a model of cell dynamics in a blood vessel to investigate the role of the cell's adhesion activity. 

%
%

\begin{figure}
\centering
\includegraphics[scale=0.35]{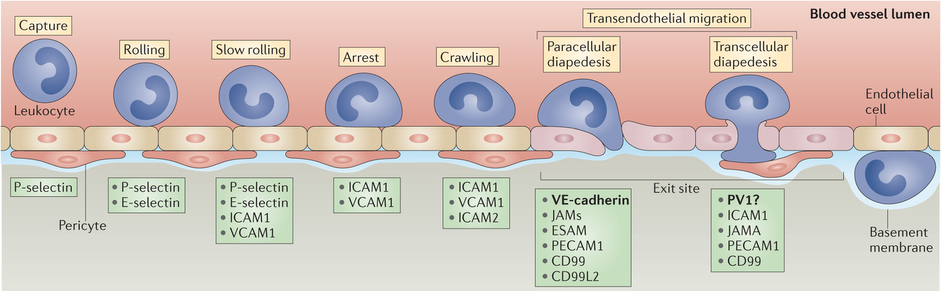} 
\caption{Scheme of the multistep cascade of leukocyte extravasation. Reprinted by permission from Macmillan Publishers Ltd: \href{http://www.nature.com/nri/index.html}{Nature Reviews Immunology} \cite{rolling}, copyright (2007).}\label{fig:leuco}
\end{figure}

\subsection{Existing Models}

There already exist several models and approaches that aim at studying cell adhesion to the vascular wall. The simplest approach consists in ignoring the cell spatial structure and in  
assuming a step-wise, stop-and-go motion. In the work by \cite{Zhao},
the dynamics of the center of mass of the cell follows a stepping process: its trajectory is approximated by a series of steps made when a cluster of bonds dissociates. Two random variables are used to describe the average stepsize and lifetime of bond clusters resisting the applied fluid force. Performing some mean field approximation makes it possible to heuristically derive a Fokker-Planck equation that governs the cell velocity evolution.
This approach allowed to explain the dispersion of rolling velocity data acquired under different experimental conditions. In the same spirit, in the absence of blood flow, macroscopic models have been developed for cell adhesion force \citep{preziosi_vitale}, where bonds are described as a distribution function. Its dynamics follows a maturation-rupture equation, also called renewal equation. In the limit of large bonds turnover, a macroscopic friction coefficient can be computed \citep{milivsic2011asymptotic,milisic2015structured}.

In the work by \cite{2017_GMMN_JTB}, the cell is described as a point carried by the blood flow and interacting with the endothelium, seen as a straight line. At the level of the individual receptor molecule, ligand binding and dissociation are stochastic Poisson processes. The bond forces are described by linear elastic forces. For constant binding and dissociation rates, the averaged problem writes as a deterministic linear Volterra integro-differential equation similar to the one considered by \cite{preziosi_vitale,milivsic2011asymptotic}, and provides some information on the cell location. Linear continuous models are not satisfactory as they can not describe the strong dependency of the cell arrest on the shear flow. In the work by \cite{2017_GMMN_JTB}, force-dependent binding rates and/or nonlinear elastic laws are also considered. For these nonlinear models a threshold on the blood velocity, under which the cell velocity vanishes, was numerically observed. Furthermore the link between the nonlinear stochastic processes and the deterministic equations was only highlighted numerically.

In the works by \cite{hammer_lauffenburger,hammer_apte}, the authors used a numerical model to describe the interplay between hydrodynamic transport and adhesion. In these works, the cell is modeled as a receptor-bearing hard sphere that interacts with a planar ligand-bearing wall. The ligand-receptor binding follows a chemical kinetic dynamics according to Bell's law, \citep{bell}. Bonds then exert elastic forces on the sphere while the linear shear flow exerts both hydrodynamic force and torque. In the PhD thesis of \cite{korn}, the Brownian motion of the sphere is taken into account in order to model the spatial receptor-ligand encountering more precisely. These models allow the numerical study of the influence of weak bonds on the motion of the sphere. More recently, other mechanical models were built to describe more precisely the hydrodynamic forces exerted on the cell \citep{reboux2007bond,efremov2011bistability,li2018rolling}. These works justify a bistable behaviour between either cell release and rolling, or between cell arrest and rolling.

\subsection{Results of the paper}

The main goal of this article is to study the balance between shear forces
and resistive adhesion forces, and its consequences on the dynamics of a cell located near
the arterial wall. To this aim, we build a stochastic model of birth-death type for a particle cell that develops adhesive interaction with the vascular wall. Note that in this setting, no information on the repartition of bonds can be obtained, so that cell rolling can not be described and will thus be confounded  with cell sliding in the following. In order to model the effect of the cell motion on the adhesion dynamics, we consider cell velocity-dependent binding rates. In such a case the model is nonlinear. 
We want to capture the transition from cell rolling to either its stopping or its release in the blood flow. For this purpose, we perform and justify some scaling limits to quantify for the effect of the cell adhesion activity on the dynamics of the cell. 

Following the approach of \cite{2017_GMMN_JTB}, we consider a minimal discrete stochastic model where the cell is a point particle carried by the blood flow in a one-dimensional setting. The adhesion dynamics is modeled by a Markovian Jump process for the formation and disassembly of bonds exerting units of resistive force. The choice of a stochastic model follows biological observations of e.g \cite{Springer}. This model has some similarities with the one heuristically derived by \cite{Zhao} but describes more precisely the interaction with the endothelium, in the spirit of \cite{hammer_apte}. Moreover, by considering constant binding forces instead of elastic ones, the present model differs from the one given by \cite{2017_GMMN_JTB}. Additionally, we take into account both the adhesions growth, modelled by the reproduction of bonds, and the effect of the cell velocity on the binding dynamics, thus leading to a nonlinear stochastic process. 

The timescale of the binding dynamics being actually fast compared to the timescale of cell motion, it is natural to use a time-continuous description of the adhesion dynamics. Therefore, in a second step, we rescale the process in the spirit of \cite{joffe_weak_1986,ethier2009markov}, and we rigorously derive continuous limiting models for the cellular adhesion dynamics. Depending on the renormalization assumptions, we obtain either a deterministic or a stochastic model, which we both study.

The deterministic model predicts that when the shear rate is high or the vascular wall is in a lowly inflamed state (with a low density of adhesion proteins), the cell develops no bonds with the wall, while it can slow down and slide on the wall otherwise, until it eventually stops and adheres to the wall. More generally, the model successfully predicts bistable behaviours in some parameter spaces, between either cell release and sliding or cell sliding and arrest. Consequently, the model is indeed able to describe the two-step process involved in the cell adhesion to the vascular wall. The analysis of the continuous stochastic model is made in two steps. First, when there is no feedback from the cell velocity on the adhesion activity, the model writes as a Cox-Ingersoll-Ross (CIR) process, for which the probability density of the arrest time is explicitly given and numerically shows a transition between the cell stopping and its release in the bloodstream. Then, in the nonlinear case, when the blood flow exerts a feedback on the binding dynamics, we derive the cell mean stopping time using a Fokker-Planck approach. 

We believe that this work can have strong implications for the immune response, drug delivery systems, as well as tumor invasion. More precisely, our model could be used as a first step in the construction of a permeability law for the vascular wall. In a future work, it will be compared with experimental in vivo data obtained by \cite{follain2018hemodynamic} in the context of metastatic invasion.

The plan of this article is the following. In Section 2, we detail the construction of the discrete stochastic model of the individual bond dynamics, and we perform its analysis together with numerical simulations. In Section 3, we proceed to rigorous derivations and study of models of the continuous time-evolution of the cell velocity, which are either deterministic or stochastic. In Section 4 we discuss our results and how they allow the characterization of the long term cell motion.


\section{A Markovian Jump process for the cell adhesion dynamics}\label{sto}

We first present the stochastic model used to investigate the adhesion process outlined in the introduction. We use a classical birth-and-death process for bonds turnover. As in most of the literature, see \cite{bell}, it is assumed that the bond lifetime has an exponential distribution. Then, we study its mathematical properties and perform some numerical simulations. 

\subsection{Modelling approach}

Let us consider a cell carried by the blood flow. We suppose that the distance between the cell and the blood vessel wall is small enough 
so that bonds may always form.  Since the cell is in the vicinity of the wall, we assume that the blood shear flow is 1D, parallel to the vascular wall and with a constant velocity, denoted by $u\in \R_+$. 

Deformability may play a role in the cell dynamics inside the blood flow and in its interaction with the vessel wall. However, some large cells like Circulating Tumor Cells can be very stiff \citep{follain2018hemodynamic}, and still follow the behaviour studied in this paper. Therefore, we hereafter choose to neglect cell deformability and to focus on the interaction between the intracellular adhesion dynamics and the vascular wall. 

In previous studies, see \cite{Hammer,Jadhav} e.g., it was shown that approximating the contact surface by a simple geometrical figure (a circle or a rectangle) and neglecting the increase of the contact surface with the flow shear rate due to cell deformability do not qualitatively change  the analysis.
 Moreover, as suggested in \cite{Bhatia}, the cell adhesion is primarily determined by physicochemical properties of adhesion proteins. In a first approximation, we thus assume  the cell to be a point particle whose position at time $t\geq 0$ is denoted by $X_t$.

\subsubsection*{Velocity model}
To describe the cell motion, we use a non-inertial approximation. Indeed, in a regime of low Reynolds number, viscous forces outweight inertial forces and the momentum equation reduces to the force balance principle: 
\begin{equation*}
V_t  = u - \gamma F_t,
\end{equation*}
where $V_{t} \in \R$ is the cell velocity, $u$ is the blood shear flow and the cell is subjected to a macroscopic resistive force, denoted by $F_t\in \R_+$, and induced by the bonds that contribute to decelerating the cell, see Figure 1. The parameter $\gamma$ is such that $\gamma^{-1}$ is a friction coefficient, following a linear force-velocity relation. The previous equation is valid only for $\gamma F_t\leq u$, as for a maximal force the cell stops and the model is no longer valid. The resistive force arises from the strength of the cell adhesion to the vessel wall. Cellular adhesion is the macroscopic readout of the forces exerted by the wall on the cell through each bond \citep{boettiger2007quantitative}. As a result, we assume that 
\begin{displaymath}
F_t = f N_t\,,
\end{displaymath}
where $f \geq 0$ is the typical force generated by a stabilized bond, and $N_t$ is the number of stabilized bonds at time $t$. Note that $f$ ranges in $\SIunits{\pico \newton}$ \citep{boettiger2007quantitative}, but its precise value depends on the experimental conditions. 

\subsection*{Non-dimensionalization}
We now introduce typical quantities for our problem: the typical force considered is the one generated by a stabilized bond: $ \overline{F}=f$ ; the associated typical velocity writes $\overline{V}=\gamma \overline{F}= \gamma f$ ; and the typical timescale of the model is taken to be the typical lifetime of a bond \citep{alon1995lifetime}: $\overline{T}=1 \SIunits{\second}$. Therefore, keeping the same notations for simplicity, the nondimensionalized problem writes 

\begin{equation*}
V_t = u -  N_t\,.
\end{equation*}
We now construct the process $(N_t)_t$ of the number of stabilized bonds over time. 

\subsubsection*{Stochastic model for the adhesive force}
Let us now introduce the simple discrete model that we use to describe the individual bonds dynamics. We write $(N_t)_t$ the Markovian processes for the number of stabilized bonds at time $t$, which follows a classical birth-and-death-like dynamics in the state space $\mathbb{N}$. 
\begin{itemize}
\itemb New bonds form spontaneously at rate $c(u)=c \mathds{1}_{u \leq u^*}$, for a velocity threshold $u^*$ above which no new bonds can be created, due to the high blood velocity.
\itemb Each existing bond can reproduce at constant rate $r$. This phenomenon reflects the local reinforcement of the connection to the vessel wall by the involvement of integrins in adhesion growth, which can be imputed to cytoskeletal forces or external stresses \citep{Geiger}. Moreover, intuitively, if an adhesion complex is composed of a large number of bonds, the unbounded molecules can find an attachment more easily than a less stable adhesion formed of fewer molecules. 
\itemb Each bond dissociates at the velocity-dependent rate $d(V_t)= d e^{\alpha V_t} = d e^{\alpha(u-N_t)}$. We choose here an exponential relation, where $d$ is the unstressed-bonds dissociation rate, and $\alpha$ is a sensitivity parameter. This choice accounts for the fact that the average lifetime of an adhesion site changes with the applied tension exerted by the blood flow: the faster the cell, the shorter the bonds' lifetime. Note that since the cell velocity is bounded by $u$, the dissociation rate is also bounded. In the following, we will write indifferently $d(V_t)$ or $d(N_t)$.
\end{itemize}

\begin{remark}
The rate of a single bond formation is mainly determined by the time spent by the two proteins near one another. Therefore, the rate $c$ should depend on the cell velocity when it is non zero. A more realistic choice for $c$ would be a decreasing function of the instantaneous cell velocity $V_t$ such as $c(v) = \left(u_* - v\right)_+ $, where $\left(\cdot \right)_+$ denotes the positive part function. Since $V_t$ depends on $N_t$, the rate rewrites $c(n) = \left(u_* - u+ n\right)_+ $ with $n$ the number of stabilized bonds. For such a choice, and assuming that $v\leq u_*$, the dependency on $n$ then only provides an additional contribution to the reproduction rate.  
\end{remark}
Note that these rates are also representative of the adhesive properties of the endothelial cells forming the vessel wall. The key point here is that there is a feedback loop between the instantaneous cell velocity and the bonds dynamics.  More elaborate dependency could be considered (see e.g \cite{milivsic2011asymptotic,milisic2015structured}), in particular involving age dependencies to model the bond elasticity (see \cite{2017_GMMN_JTB}),
but we choose to keep a minimal set of parameters, for both simplicity and the sake of clarity.


The balance between the adhesion force on the one hand, and the load and torque created by the blood flow on the other hand then determines the outcome of the dynamics: either the cell rolling and arrest or its release in the blood flow. 

\subsection{Mathematical properties of the discrete model}
In this section we derive some mathematical properties based on stochastic analysis tools, and we perform numerical simulations of the process $(N_t)_t$. The dissociation rate being nonlinear, classical tools do not apply. Since we are interested in the dynamics while $V_t \geq 0$, that is, while $N_t \leq u$, we define the stopping time
\begin{displaymath}
\tau_u := \inf_{t \geq 0} \{ N_t \geq u\}\,.
\end{displaymath}
We are interested in the Markovian jump process $(N_t)_{t\in [0,\tau_u]}$ defined by the following transitions:
\begin{equation}\label{eq:taux}
n \mapsto \left\lbrace
\begin{array}{ll}
n+1 & \text{ at rate } \lambda(n) = c(u) + r n \, , \\
n-1 & \text{ at rate } \mu(n) = d(n) n\,,  \\
\end{array}\right.
\end{equation}
where $\lambda$ and $\mu$ are defined on $\N$, and are bounded. Such a process is classically well-defined (see e.g \cite{ethier2009markov}). We also control the mean number of bonds in finite time. 

\begin{proposition}[Moments propagation]\label{momentprop1Ddiscret}
Assume that there exists $p \geq 1$ such that $\mathbb{E}\left[N_0^p\right] <+\infty$. Then,
\begin{displaymath}
\mathbb{E}\left[\sup_{t\in [0,T \wedge \tau_u]}N_t^p\right] <+\infty\, \forall \; T >0\,.
\end{displaymath}
\end{proposition}

\begin{proof}
This Proposition is proved in a more general framework in Appendix \ref{annex1}.
\end{proof}
The mean path of this process can not be fully studied in the general case. Indeed, for $\E\left[N_0\right] < + \infty$, we classically write the mean equation 
\begin{equation}\label{eq:lin}
 \E\left[N_{t\wedge \tau_u}\right]= \E\left[N_0\right] +  c(u) \E[t \wedge \tau_u]+\E \left[ \int_{0}^{t\wedge \tau_u} \left( r-de^{\alpha(u-N_s)} \right)N_s  \Dx s\right]\,
\end{equation}
where, even when assuming $t\leq \tau_u$, the nonlinearity prevents any analysis. In a simpler case, namely when there is no feedback from the cell velocity on the adhesion dynamics, and without considering the stopping time $\tau_u$, we obtain a classical immigration-birth-death process, that was already studied by \cite{Bansaye2015,dessalles2018exact}. More precisely, $(N_t)_t$ then follows a negative binomial distribution of parameters $\left(\frac{c(u)}{r},\frac{r}{d}\right)$. It follows that

\begin{equation}\label{eq:meanODE}
\E[N_t] = \left\lbrace 
\begin{array}{ll}
\E[N_0] + c(u)   t  & \text{ if }r = d\, , \\
\E[N_0] e^{(r-d)t} + \frac{c(u)}{r-d} \left(e^{(r-d)t} -1 \right) & \text{ otherwise.}
\end{array}\right.
\end{equation}
At steady state, one finds 
\begin{displaymath}
\E[N]_\infty := \left\lbrace 
\begin{array}{ll}
\frac{c(u)}{d-r}  & \text{ for }r < d\, , \\
+\infty & \text{ otherwise,}
\end{array}\right.
\Leftrightarrow
\E[V]_\infty := \left\lbrace 
\begin{array}{ll}
u -  \frac{c(u)}{d-r}  & \text{ for }r < d\, , \\
-\infty & \text{ otherwise,}
\end{array}\right.
\end{displaymath}
with 
\begin{equation*}
\text{Var}(N_{\infty}) = \frac{c(u)}{d\left(1-\frac{r}{d} \right)^2}\,.
\end{equation*}
In the case of a circulating cell, we are only interested in the situation where $v\leq u$. As a consequence, assuming $\E[N_0]=0$ and $u>0$, we obtain the following mean asymptotic behaviours:
\begin{center}
\begin{tabular}{l|l||c|c}
\multirow{2}{*}{$u>u_*$} &  & $v=u$ &\textbf{Cell release} \\
\multirow{2}{*}{$c(u)=0$} & & &\\
 &&\multicolumn{2}{r}{}\\ \hline \hline
\multirow{2}{*}{$u\leq u_*$} &\multirow{4}{*}{$r<d$} & \multirow{2}{*}{$\displaystyle0< \frac{c}{d-r}<u$} & \multirow{2}{*}{\textbf{Cell sliding}} \\
\multirow{2}{*}{$c(u)=c>0$}&&&\\
\cline{3-4}
&& \multirow{2}{*}{$\displaystyle 0< u \leq \frac{c}{d-r}$} & \multirow{2}{*}{\textbf{Cell arrest}} \\ 
&&&\\
\cline{2-4}
&\multirow{2}{*}{$r\geq d$} & \multicolumn{2}{r}{\multirow{2}{*}{\textbf{Cell arrest}}}\\
&&\multicolumn{2}{r}{}
\end{tabular}
\end{center}
These results show that the birth-and-death dynamics without feedback intrinsiquely carries a dichotomic asymptotic behaviour. However, such a model is not fully satisfactory, since the only shear-treshold effect that appears concerns the initiation of the adhesive interaction and not the cell fate.

\subsection{Numerical Simulations}
The process $(N_t)_t$ being Markovian, it can be simulated directly events after events. Consider the population size $N_{T_k}$ at time $T_k$. Then, 
\begin{itemize}
\itemb the global jump rate writes $\varsigma_k = \lambda(N_{T_k})+\mu(N_{T_k})$, which means that the time before the next event is a random variable following an exponential law of parameter $\varsigma_k$. A realization of this law provides $T_{k+1}$. 
\itemb a new bond is created with probability $\lambda(N_{T_k})/ \varsigma_k$, while with probability $\mu(N_{T_k}) / \varsigma_k$ a randomly chosen bond disassembles, and $N_{T_{k+1}}$ follows.
\end{itemize}
The above procedure can be iterated to give the time evolution of the process. Numerical simulations of the process are displayed in Figure \ref{fig:discrete}. One may observe that the velocity may either shrink to zero or remain close to $u$ for the same parameter values. Note also that sliding phases are observed in both cases.

\begin{figure}
\captionsetup[subfigure]{justification=centering}
	\begin{center}
	\subfloat[The stochasticity induces a sliding phase that ends up in the cell arrest.]{
      \includegraphics[width=0.45\textwidth]{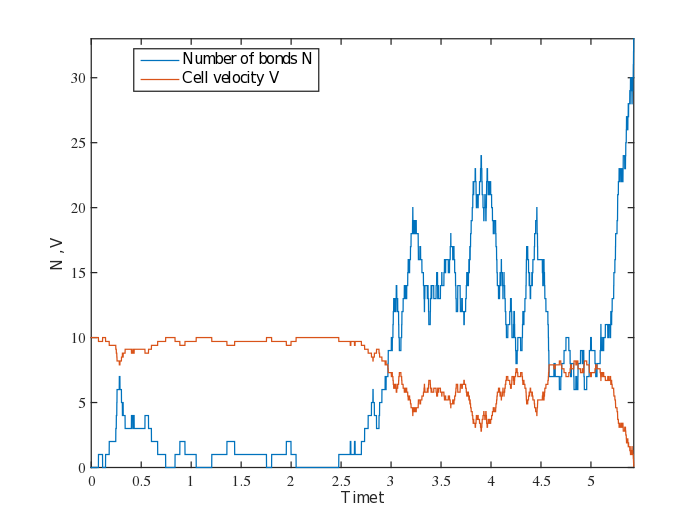} 
                         }
     	\subfloat[The cell only experiences a small sliding phase that does not prevent its release in the blood flow.]{
      \includegraphics[width=0.45\textwidth]{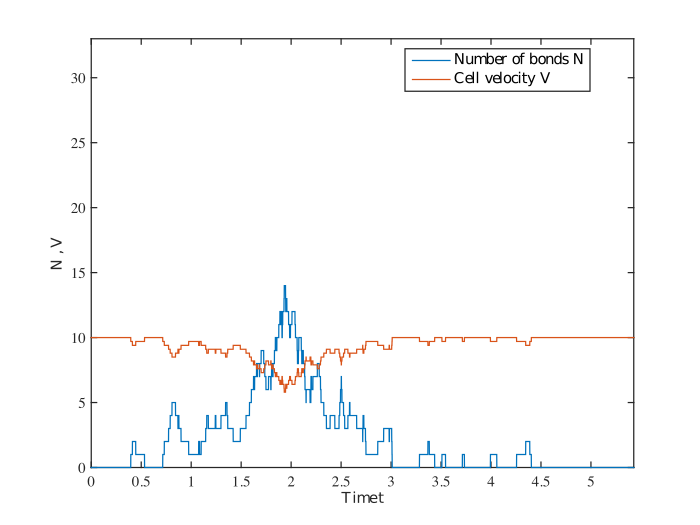}
                         }     
\caption{Numerical simulations of the discrete process defined by \eqref{eq:taux}. Parameters: $(u,c,r,d,\alpha)=(33,4,5,3,0.1)$.}\label{fig:discrete}
	\end{center}
\end{figure}


\section{Continuous limiting models and characterization of the dynamics}\label{cont}

 In this section we separate the scale of the adhesion dynamics from the one of the cell motion. Such scale separation is justified by the large number of bonds, and by the very fast binding dynamics, as compared to the cell displacement. As an illustration, a bond lifetime ranges around $1\SIunits{\second}$, a typical binding rate ranges aroung $10^3$ $\SIunits{\reciprocal \second}$ whereas the cell rolling velocity ranges around $30 \SIunits{\micro\meter  \reciprocal \second}$ (see \cite{2017_GMMN_JTB} and the references therein). As will be seen next, this assumption allows the use of a scaling approach to derive two continuous limiting models, for which deeper analysis can be pursued.

More precisely, 
let $K\geq 1$ be a parameter such that $1/K$ scales the force generated by a fraction of bonds.
Moreover, we assume that the bond dynamics gets faster and faster as the fraction of bonds gets smaller: we consider now $K$-dependent rates $c^K,\; r^K,$ and $ d^K$, related to the process $(N_t^K)_t$. We then define the renormalized process $(Z_t^K)_t$ by 
 \begin{equation}\label{proc_renorm1D}
 Z_t^K = \frac{1}{K} N_t^K \in \frac{1}{K}\N.
 \end{equation}

\subsection{Deterministic continuous limiting model}
We consider the following rates: 
\begin{equation}\label{rappel_rates}
 c^K(u) = K c(u)\, , \;  r^K= r\, , \textrm{ and } d^K(KZ^K_t) = d(Z^K_t)\, .
\end{equation}
\noindent
In other words, in considering an approximation for a continuous description of bonds, the formation rate is larger, while the self-enhancement of the adhesion dynamics and the bonds typical lifetime stay unchanged.
From the modelling viewpoint, such an assumption amounts to considering adhesion clusters that involve only a small number of  proteins on each side (wall and cell). 
Note that since the reproduction is not accelerated, the clustering that leads to adhesion growth is not large enough to induce stochastic fluctuations at the cell level. 
 In this context, we obtain the following convergence result.

\begin{theorem}\label{lim_deterministe_1D}
Consider the sequence of processes $(Z^K)_K$ for $(Z^K_t)_{t\geq 0}$ defined by \eqref{proc_renorm1D} and rates defined by \eqref{rappel_rates}.
 If $Z_0^K \underset{K\rightarrow + \infty}{\longrightarrow} n_0\in \R_+$ in probability, and if 
\begin{equation*}
\sup_{K >0} \mathbb{E}\left[(Z_0^K)^2\right]<+ \infty\, ,
\end{equation*}
then, for $T>0$, $(Z^K)_{K>0}$ converges in law in $\mathbb{D}\left([0,T],\R_+ \right)$ to the unique continuous function $n \in \mathcal{C}([0,T],\R_+)$ solution to
\begin{equation}\label{limite1_eq1D}
n(t) = n_0 + \int_0^t c(u) + ( r - d(n(s)) ) n(s) \Dx s \, .
\end{equation}
\end{theorem}

\begin{remark}
By the Gronwall lemma, one has for $T < \infty$,
\begin{equation*}
\sup_{t\in [0,T]} n(t) \leq  (n_0 + c T) e^{r T}<+\infty\, ,
\end{equation*}
showing that the global density stays finite in finite time. 
\end{remark}

\begin{proof}\label{preuve_deterministe1D}
The proof is displayed in Appendix \ref{app:det_conv}. 
\end{proof}
We now perform the analysis of the limiting problem. Let us define the function $F$ by $F(n)=c\mathds{1}_{u\le u_*}+ \left( r -de^{\alpha (u- n)}\right) n$. We prove the following result.

\begin{proposition}\label{prop:det_asympt}
Assume that the rates are given by \eqref{rappel_rates}. Then the stationary state(s) $n^\infty$ of \eqref{limite1_eq1D} are as follows.
\begin{enumerate}
\item If $u> u_*$, then the system admits two stationary states  $n^\infty_1= 0$ and $n^\infty_2 =u- \frac{1}{\alpha }\ln (\frac{r}{d}) $. The smallest is stable and the largest is unstable. 
\item If $u\le u_*$, 
\begin{enumerate}
\item for $u\le  \frac{1}{\alpha}\ln \left( \frac{r}{d}\right)$, then  $n^\infty = +\infty$. 
\item for $u> \frac{1}{\alpha}\ln \left( \frac{r}{d}\right)$, then there exists a unique $0<\bar n< \frac{1}{\alpha}$ such that $F'(\bar{n})=0$ and\begin{enumerate}
\item If $F(\bar{n})>0$, then  $n^\infty = +\infty$.
\item If $F(\bar{n})=0$, then $\bar{n}$ is the unique stationary solution.
\item If $F(\bar{n})<0$, then there exists two stationary solutions $n^\infty_1$ and $n^\infty _2$, such that $ 0< n^\infty_1 < \bar{n} <n^\infty_2< + \infty$, the smallest being stable and the largest unstable. 
\end{enumerate}
\end{enumerate}
\end{enumerate}
\end{proposition}

\begin{proof} The case $u> u_*$ follows from a direct computation. Consider the case where $u\le  u_*$, then one has 
\begin{equation*}
n'(t) =  c+ \left( r -de^{\alpha (u- n(t))}\right) n(t)= F(n(t))\,.
\end{equation*}
To understand the dynamics of the velocity of the cell, we have to understand
that of $n$. We are therefore interested in the values of $n$ for which $F(n)=0$, which are the steady states, and in the sign of $F$ required to obtain the stability of the steady states.
A quick computation shows that 
\begin{eqnarray*}
F'(n) &=& r +d\left(\alpha  n -1\right) e^{\alpha (u- n)}\,, \\
F''(n) &=&  \alpha  d \left(2-\alpha  n\right) e^{\alpha (u- n)} \,.
\end{eqnarray*}
We can study the sign of $F''(n)$ to obtain the variations of $F'$. This yields the following variation table:
\begin{center}
\begin{variations}
n	&	\z	&		&\frac{2}{\alpha}&		& \pI	 \\ \filet
F''(n)&		& +		&	\z	&	-	&		 \\ \filet
\m {F'(n)} &r -de^{\alpha u} 	&	\c	& \h {F'(\frac{2}{\alpha})>0}&	\d	&  r	  \\
\end{variations}
\end{center}
Since $F'(2/\alpha)>0$ and $r>0$, the existence of a stationary state $n^{\infty}$ such that $F(n^{\infty})=0$ depends on the sign of $F'(0)=r-de^{\alpha u}$. As a consequence, 
\begin{itemize}
\item If $u\le \frac{1}{\alpha}\ln \left( \frac{r}{d}\right)$, then $\forall n\in \R_+$, $F(n)\ge c>0$, hence $n^\infty = +\infty$. 
\item If $u> \frac{1}{\alpha}\ln \left( \frac{r}{d}\right)$, then $F'(0)<0$, and there exists a unique $0<\bar n< \frac{2}{\alpha }$ such that $F'(\bar{n})=0$. This provides the sign of $F'$, from which we obtain the variation table for $F$ that leads to the result:
\begin{center}
\begin{variations}
n	&	\z	&		&\bar n <\frac{2}{\alpha }&		& \pI	 \\ \filet
F'(n)&		& -		&	\z	&	+	&		 \\ \filet
\m {F(n)} &\h {c>0} 	&	\d	& F(\bar n)&	\c	&  \h {\pI}	  \\
\end{variations}
\end{center}
\end{itemize}
Notice that since $u> \frac{1}{\alpha}\ln \left( \frac{r}{d}\right) \Leftrightarrow de^{\alpha u} > r$, both behaviours arise according to the comparison between the reproduction and death rates. Note also that since $F'$ is strictly increasing on $\left(0,\frac{2}{\alpha}\right)$ and since $F'(1/\alpha) = r>0$, we obtain that $\overline{n}<1/\alpha$.
\end{proof} 

\begin{remark}
Note that if  $u_*\ge u> \frac{1}{\alpha}\ln \left( \frac{r}{d}\right)$ the three cases described in the proposition above may occur. Indeed, consider the particular case where $d>r$ and $u=\frac{1-\frac{r}{d}}{\alpha}$, then $\bar n= u$ and $F(\bar n)=c-\frac{(d-r)^2}{\alpha d }$ whose sign depends on the value of $c$. The dynamics is then dependent on the ability of the cell to form bonds at primary contact.
\end{remark}
We are now interested in the biological interpretation of this study. Since our model is only valid up until the cell adhesion to the vessel wall, we consider Equation \eqref{limite1_eq1D} up until the adhesion density reaches $u$. The following Corollary locates $u$ with respect to the stationary state(s) of the system. This allows us to assess the cell fate depending on the parameter values. For $\alpha>0$, denote the key parameters
\begin{equation*}
U_{\alpha}:=\frac{1}{\alpha}\ln(r/d), \qquad \overline{U}_{\alpha} := \frac{1}{\alpha}(1-r/d), \qquad U_c := \frac{c}{d-r}, \text{ and }  \overline{C}:=\frac{1}{\alpha d} (r-d)^2\,.
\end{equation*}

\begin{corollary}\label{coro:det}
Let $n_0 \in [0,u]$ with $u>0$. Then, Equation \eqref{limite1_eq1D} admits either one stationary state denoted by $n^\infty$, or two stationary states $n_{1,2}^{\infty}$ such that $0< n_{1}^{\infty}<n_{2}^{\infty}$, the smaller one being stable and the larger unstable. 
\begin{enumerate}
\item If $u> u_*$:
\begin{enumerate}
\item for $r\leq d$, then $n^\infty _1= 0$ and $n^\infty _2 =u- U_{\alpha} \geq u$. 
\item for $r>d$, if $u \leq U_{\alpha}$, $n^\infty=+\infty$ ; if $u > U_{\alpha}$, one has $n^\infty_1=0< n^\infty _2 =u- U_{\alpha} < u$.
\end{enumerate}
\item If $u\leq u_*$, then 
\begin{enumerate}
\item for $r< d$, $\exists ! 0<\bar n< \frac{1}{\alpha}$ such that $F'(\bar{n})=0$. 
\begin{enumerate}
\item For $u > \overline{U}_{\alpha}$, $F'(u)>0$, so that $u>\overline{n}$.
\begin{enumerate}
\item If $u> U_c$, then we have $0<n_1^{\infty} < \overline{n}<u<n_2^{\infty}$. 
\item If $u= U_c$, then $0<n_1^{\infty} <n_2^{\infty}=u$.
\item If $ u < U_c $, then if $F(\bar n) <0$, one has $0<n_1^{\infty} < \overline{n}<n_2^{\infty}<u$ ; if $F(\bar n)=0$, then $n^{\infty}=\overline{n}<u$ ; if $F(\bar n)>0$, then $n^{\infty}=+\infty$.
\end{enumerate}
\item For $u = \overline{U}_{\alpha}$, then $u=\overline{n}$. 
\begin{enumerate}
\item If $c> \overline{C}$, then $n^{\infty}=+\infty$.
\item If $c= \overline{C}$, then $n^{\infty}=u$.
\item If $c<\overline{C}$, then $0< n_1^{\infty} < u < n_2^{\infty}$.
\end{enumerate}
\item For $u < \overline{U}_{\alpha}$, then $0<u< \overline{n} $.
\begin{enumerate}
\item If $u= U_c$, then $n_1^{\infty}=u<n_2^{\infty}$.
\item Otherwise, if $F(\bar n)=0$, then $u < n^{\infty}=\overline{n}$ ; if $F(\bar n)>0$, then $n^{\infty}=+\infty$. Finally, if $F(\bar n)<0$, then  $u< n_1^{\infty}<n_2^{\infty}$ when $u< U_c$, and $n_1^{\infty}<u< \overline{n}$ when $u>U_c$.
\end{enumerate}
\end{enumerate}
\item For $r=d$, there exists a unique $0<\bar n< \frac{1}{\alpha}$ such that $F'(\bar{n})=0$ and $u>\overline{n}$. If $F(\bar n) <0$, we have $0< n_1^{\infty} < \overline{n} < n_2^{\infty}<u$. If $F(\bar n)=0$, $n^{\infty}=\overline{n}<u$. If $F(\bar n ) >0$, $n^{\infty}=+\infty$. 
\item For $r>d$, $F(u)>0$.
\begin{enumerate}
\item if $u \leq U_{\alpha}$, then $n^{\infty}=+\infty$.
\item if $u>U_{\alpha}$, then $\exists ! 0<\bar n< \frac{1}{\alpha}$ such that $F'(\bar{n})=0$ and $u> \overline{n}$. Then, if $F(\bar n) <0$, we have $0< n_1^{\infty} < \overline{n} < n_2^{\infty}<u$. If $F(\bar n)=0$, $n^{\infty}=\overline{n}<u$. If $F(\bar n ) >0$, $n^{\infty}=+\infty$.
\end{enumerate}
\end{enumerate}
\end{enumerate}
\end{corollary}

\begin{proof}
The Corollary results from Proposition \ref{prop:det_asympt} combined with the sign analysis of $F(u)$ and $F'(u)$ that depends on parameter values. 
\end{proof}

Let us give an interpretation of the above results. First, note that when there are two stationary states, the largest one denoted by $n_2^{\infty}$ is unstable. Therefore, if $n_2^{\infty}<u$, it follows that for $n(0)>n_2^{\infty}$, the adhesion density will increase until the cell adheres to the wall. This can be interpreted as another stable stationary state for the model, which therefore captures the bistability previously described in experimental and theoretical studies. This property will be further discussed in Section \ref{discussion}. 

Table \ref{table:limdet1} gives the outcome of the cell dynamics in the case where $u>u_*$. In this case, the blood flow is so fast that the cell does not initiate any new adhesion. However, the dynamics may still be interesting in the case where some bonds already exist and may be stabilized by the growth of adhesions. Then, if $r\leq d$, dissociation is always more frequent than strenghtening of existing adhesions, so that starting with $n_0\leq u$, the cell is released into the bloodstream. When $r>d$, and depending on the sensitivity $\alpha$ of the dissociation rate on the cell velocity, the cell either goes back in the bloodstream or adheres firmly to the wall, while sliding is an unstable state. 

Table \ref{table:limdet2} shows the possible outcomes in the case where $u\leq u_*$. In this situation, the whole adhesion dynamics is active. Then, the conditions discriminating between different cell fates are based on the balance between the formation of adhesion bonds (related to $r$ and $c$), and their dissociation (related to $d$ and $\alpha$). This study shows how Equation \eqref{limite1_eq1D} has more complex behaviours in comparison with the mean linear ODE \eqref{eq:meanODE}. 

Let us comment on these results. First of all, not surprisingly, our model successfully predicts the shear-stress threshold above which nor capture nor sliding occurs when the cell does not initially interact with the wall. This is explained by the regulation of bonds dissociation by the cell velocity. The model also predicts the firm adhesion to the wall that is observed experimentally. Finally, the model provides conditions on the parameter space for both stable sliding and bistable behaviour between cell sliding and cell adhesion. This results come from the competition between bonds formation and rupture that occurs in the cell-wall contact area.


\begin{figure}[h]
\centering
\subfloat[][Case $u> u_*$]{\label{table:limdet1}
\begin{tabular}{l|l||l|l|r}
  \multirow{5}{*}{$r> d$} &  \multirow{2}{*}{$u\leq U_{\alpha}$} & \multicolumn{2}{|l|}{
\multirow{2}{*}{$
n^{\infty}= \begin{cases}
+\infty \text{ if initial contact,}\\
0 \text{ otherwise.}
\end{cases}
$} }& \textbf{Cell arrest} if initial contact,\\
 & & \multicolumn{2}{|l|}{ } & \textbf{Cell release} otherwise \\
 &&\multicolumn{2}{|l|}{ }&\\
 &$u> U_{\alpha}$ &  \multicolumn{2}{|r|}{$\textcolor{blue}{n_1^{\infty}=0} <   \textcolor{red}{n_2^{\infty}=u-U_{\alpha}}<u$} & \textbf{\textcolor{blue}{Cell release}}, \textbf{\textcolor{red}{Cell sliding}} \\
 & & \multicolumn{2}{|r|}{} & and \textbf{Cell arrest} \\
\hline
 \multicolumn{2}{c||}{ \multirow{2}{*}{$r\leq d$}} & \multicolumn{2}{|r|}{\multirow{2}{*}{$\textcolor{blue}{n_1^{\infty}=0} < u<  \textcolor{red}{n_2^{\infty}=u-U_{\alpha}}$}} & \multirow{2}{*}{\textbf{\textcolor{blue}{Cell release}}} \\
\multicolumn{2}{c||}{ } & \multicolumn{2}{|r|}{}& \\
\end{tabular}
}\\
\subfloat[][Case $u\leq u_*$]{\label{table:limdet2}
\begin{tabular}{l|l|l||l|r}
 \multirow{2}{*}{$r>d$} & \multicolumn{2}{c||}{$u\leq U_{\alpha}$}  &$n^{\infty}=+\infty$ &\textbf{Cell arrest} \\
 \cline{4-5}
  & \multicolumn{2}{c||}{$u> U_{\alpha}$}  &\multirow{3}{*}{$u> \overline{n}$} &   \textbf{Cell arrest}\\
 \cline{1-3}
   \multicolumn{3}{c||}{$r=d$} & & or  \textbf{Cell sliding}\\
   \cline{1-3}
   \multirow{10}{*}{$r<d$} & \multirow{4}{*}{$u>\overline{U}_{\alpha}$} & $u< U_c$ & & or \textbf{Two sliding regimes}\\
    \cline{3-5}
   & &  \multirow{2}{*}{$u=U_c$} &  \multirow{2}{*}{$0< \textcolor{blue}{n_1^{\infty}} <  \textcolor{red}{n_2^{\infty}}<u$} &\textbf{Two sliding regimes}\\
   & & & & \textbf{and {Cell arrest}}\\
   \cline{3-5}
   & & $u>U_c$ & \multirow{2}{*}{$0< \textcolor{blue}{n_1^{\infty}} <u<  \textcolor{red}{n_2^{\infty}}$} & \multirow{2}{*}{\textbf{\textcolor{blue}{Cell sliding}} and \textbf{\textcolor{red}{Cell arrest}}}\\
  \cline{2-3}
   & \multirow{3}{*}{$u=\overline{U}_{\alpha}$} & $c< \overline{C}$ &  &\\
   \cline{3-5}
   & & $c=\overline{C}$ & $n^{\infty}=u$ &\multirow{3}{*}{\textbf{Cell arrest}}\\
     & &$c> \overline{C}$ & $n^{\infty}=+\infty$ &\\
   \cline{2-3}
    & \multirow{3}{*}{$u< \overline{U}_{\alpha}$} & $u=U_c$ & $0< \textcolor{blue}{n_1^{\infty}=u} <  \textcolor{red}{n_2^{\infty}}$ &\\
    \cline{4-5}
    & & \multirow{2}{*}{$u\neq U_c$} &\multirow{2}{*}{$0<u<\overline{n}$} & \textbf{\textcolor{blue}{Cell sliding}} and \textbf{\textcolor{red}{Cell arrest}} \\
    & & &   &or \textbf{Cell arrest}
\end{tabular}
}
\caption{Table of stationary solutions of \eqref{limite1_eq1D} and corresponding situations. Stationary solutions are in red if unstable, blue if stable.}\label{table:limdet}
\end{figure}

\subsection{Stochastic continuous limiting model}\label{sec:demography}
In this section, we consider the following rates:
\begin{equation}\label{eq:rates_sto}
c^K(u) = K c(u),\quad r^K=  r+ K a, \text{ and } d^K(KZ^K_t) = d(Z^K_t) + K a,
\end{equation}
with $K>0$ and $a>0$. The whole adhesion dynamics is therefore accelerated.
Note that using the same acceleration for reproduction and death permits to  keep the same bounded individual growth rate $r^K - d^K = r-d$. 
This way, even if each adhesion bond reproduces and dies infinitely faster, its contribution to the global adhesion growth remains the same.

\begin{theorem}\label{theo_stoch1D}
Consider the sequence of processes $(Z^K)_K$ for $(Z^K_t)_{t\geq 0}$ defined by \eqref{proc_renorm1D} and rates defined by \eqref{eq:rates_sto}.
If for $K \rightarrow + \infty$ the initial value $Z_0^K$ converges in law to a $\R_+$-valued random variable $N_0 $, with 
\begin{equation*}
\sup_{K >0} \mathbb{E}\left[(Z_0^K)^2\right]<+ \infty,
\end{equation*}
then $(Z^K)_{K>0}$ converges in law in $\mathbb{D}\left([0,T],\R_+ \right)$ to the continuous process $N=(N_t)_{t\in [0,T]} \in \mathcal{C}([0,T],\R_+)$ solution of 
\begin{equation}\label{SDE_discret}
\Dx N_t =  b(N_t)\Dx t  + \sigma(N_t)\Dx B_t ,
\end{equation}
with $B_t$ a Brownian Motion, $b(N_t) =c(u) + (r-d(N_t))N_t$ and
$\sigma(N_t) =\sqrt{2a N_t}$.
\end{theorem} 

\begin{proof}
The proof is displayed in Appendix \ref{annex:stoch}.
\end{proof}

\begin{remark}
The solution to the SDE \eqref{SDE_discret} is almost surely positive if $b(n)\geq 0$ for all $n\geq 0$, and for a positive initial state (see the 1D comparison principle in \cite{karatzas} e.g.).
\end{remark}

\paragraph{Numerical simulations}
In order to preserve the positivity of the process, we perform numerical simulations of the SDE \eqref{SDE_discret} using a symmetrized Euler scheme, which consists in taking the absolute value of the classical Euler scheme (see e.g \cite{Berkaoui}). More precisely, the scheme is the following: write $(N_k)_k$ for the discretization of $(N_t)_t$, where $N_k$ corresponds to the time $t_k = k \Delta t$. Then, define $N_0 = n_0$ and for $k \geq 0$, 
\begin{equation*}
N_{k+1} = \mid N_k  + b(N_k) \Delta t + \sqrt{2a \Delta t N_k} W \mid \,,
\end{equation*}
with $W \sim \mathcal{N}(0,1)$. The strong $L^1$ convergence of this scheme is proved by \cite{Berkaoui} provided that
\begin{displaymath}
\frac{\sigma^2}{8} \left(\frac{2b(0)}{\sigma^2} -1 \right)^2 > 3P \vee 4 \sigma^2\,,
\end{displaymath}
for $P \geq |r-d|$ and $\Delta t \leq \frac{1}{2P}$. This condition allows to deal with the non-Lipschitz diffusion coefficient, and rewrites, in our case,
\begin{displaymath}
\frac{a}{4} \left(\frac{c}{a} -1 \right)^2 > (3P \vee 8a) \,.
\end{displaymath}
As an example, $\frac{a}{4} \left(\frac{c}{a} -1 \right)^2 >  8a$ is equivalent to $\left(\frac{c}{a} -1 \right)^2 >32$, which is verified for $c>7a$. The numerical simulations are displayed in Figure \ref{fig:EDS}, and show two typical arrest and release situations.

\begin{figure}
\begin{center}
	\subfloat[Example of cell arrest, $c=4$.]{
	\includegraphics[width=0.45\textwidth]{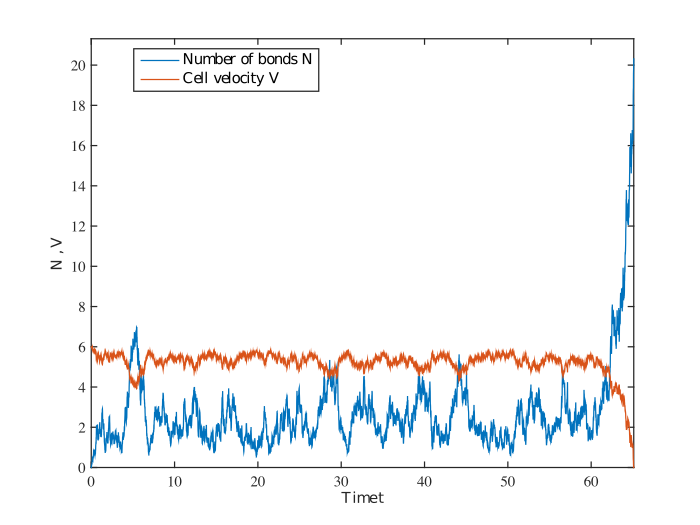} 
	}
	\subfloat[Example of cell sliding, $c=5$.]{
	\includegraphics[width=0.45\textwidth]{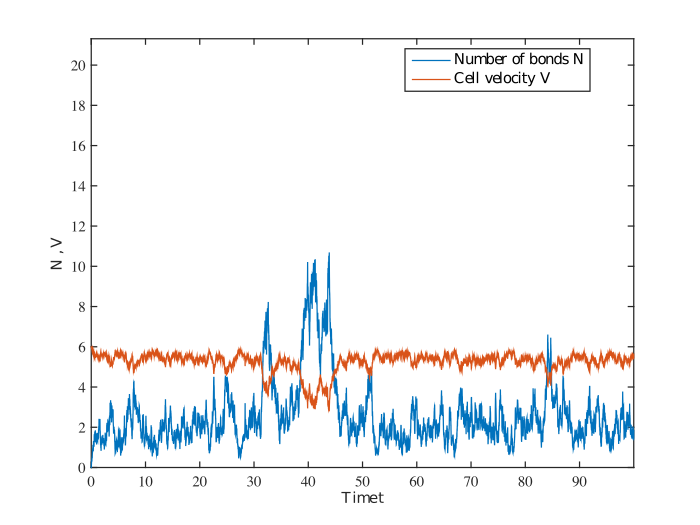} 
	}
\caption{Numerical simulations of the solution of the SDE \eqref{SDE_discret}, showing typical arrest (left) and sliding before release (right) situations. The cell velocity is displayed in red while the cell position is displayed in blue. Parameters: $(u,r,d,\alpha,a)=(20,5,4,0.1,0.55)$. }\label{fig:EDS}
\end{center}
\end{figure}


The rolling motion of individual cells has been observed to fluctuate randomly both in vivo and in vitro. The model \eqref{SDE_discret} obtained in Theorem \ref{theo_stoch1D} is therefore suited to investigate the effects of the stochastic fluctuations associated with the adhesion activity on the cell dynamics. In particular, we are interested in the expected time until the cell firmly adheres to the vessel wall. 

\subsubsection*{The case without feedback: the CIR process.}

When we assume that the cell velocity exerts no feedback on the bonds dissociation, the solution of \eqref{SDE_discret} reduces to a CIR process, see \cite{GoingJaeschke1999,StephenShreve,ikeda}:
\begin{equation}\label{eq:CIR}
\Dx N_t = (c+ (r-d)N_t) \Dx t + \sqrt{2a N_t } \Dx B_t\,,
\end{equation}
with $c>0$, $a>0$ and $r-d \in \R$. It is known that such processes arise from the diffusion limit of discrete branching processes with immigration \citep{athreya1972branching}, and show a dichotomy behaviour. Depending on the parameters, the density is almost surely either close to $0$ or large, leading to the almost sure cell arrest in our model. Simulations of this process are displayed in Figure \ref{fig:CIR}. Some general properties of the CIR process are displayed in Appendix \ref{annex:CIR}. In particular, its stationary probability density is represented in Figure \ref{fig:CIR} and shows the transition between both behaviours.

\begin{figure}[h]
\captionsetup[subfigure]{justification=centering}
  \begin{center}
    \subfloat[Subcritical case with $(c,a,r,d)=(0.5,1.5,4.45,4.5)$.][Subcritical case with \\$(c,a,r,d)=(0.5,1.5,4.45,4.5)$.]{
      \includegraphics[width=0.45\textwidth]{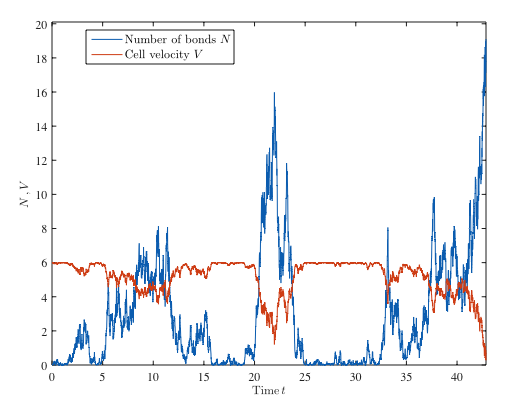}
      \label{fig:traj_cir1}
                         }
    \subfloat[Supercritical case with $(c,a,r,d)=(2,1,4,4)$.]{
      \includegraphics[width=0.45\textwidth]{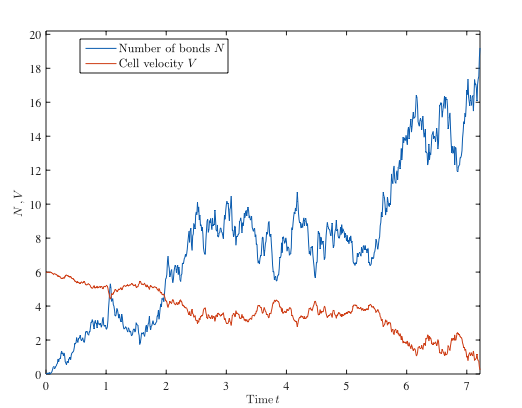}
      \label{fig:traj_cir2}
                         }  \\
    \subfloat[Subcritical case with $(c,a)=(1,2)$.]{
      \includegraphics[width=0.45\textwidth]{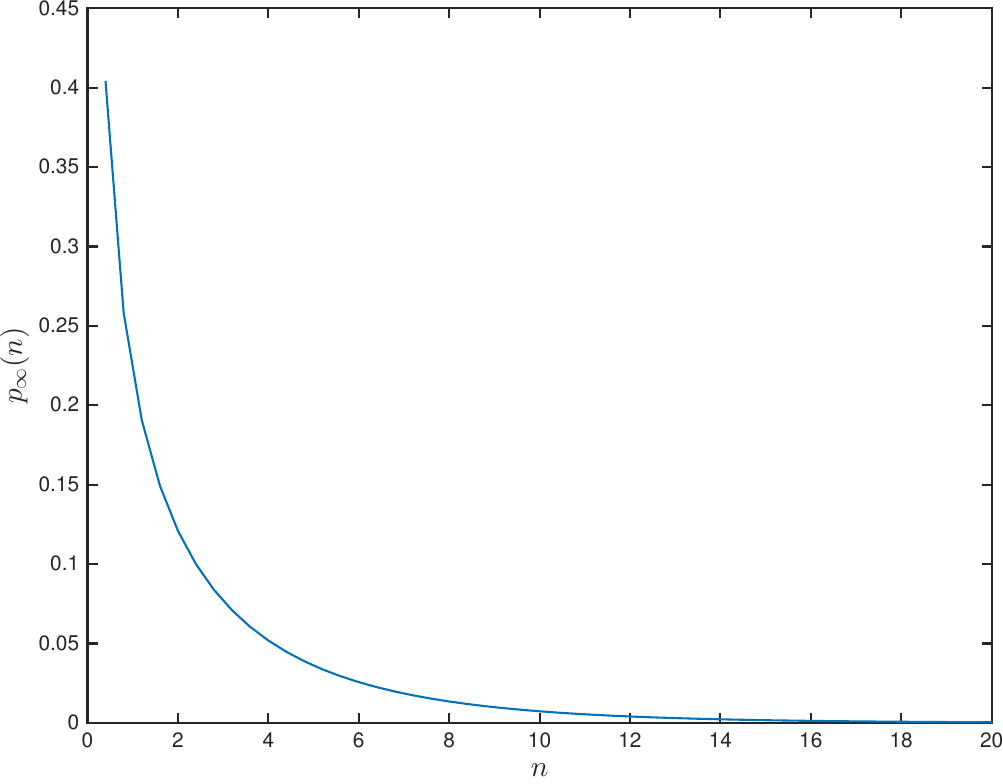}
      \label{fig:dens_cir1}
                         }
         \subfloat[Supercritical case with $(c,a)=(5,2)$.]{
      \includegraphics[width=0.45\textwidth]{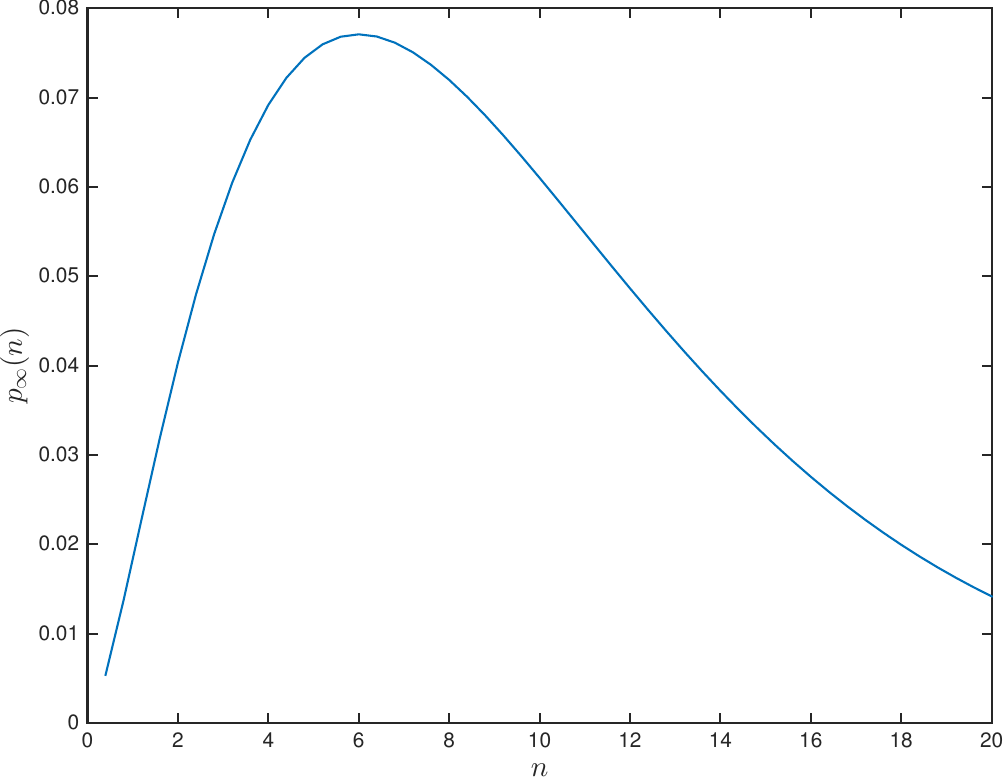}
      \label{fig:dens_cir2}
                         }
    \caption{Up: numerical simulations of the CIR process \eqref{eq:CIR}. In the subcritical case, the adhesion density almost surely reaches zero, while in the supercritical case, an adhesive interaction is almost surely sustained. Down: numerical simulations of the stationary probability density of the CIR process for $(u,\alpha,r,d)=(20,0.1,4,4.5)$.
}
    \label{fig:CIR}
  \end{center}
\end{figure}

\paragraph{Time to reach $u$:}
It is also possible to obtain information on the time needed to reach a given value. More precisely, one can use the Laplace transform of the first hitting time of any value, starting at a given point \citep{GoingJaeschke1999,leblanc1998path}. However, it is not possible to proceed to its inversion analytically. While numerical inversions procedures exist (some of them are compared by \cite{leblanc1998path}), they do not always provide satisfactory results: the integral of the output may not be equal to one, and negative values may appear. The procedure proposed by \cite{abate1992fourier} seems satisfactory in this viewpoint. \par 
In this paper, we follow the work of \cite{linetsky04,linetskyb} to numerically compute the first hitting time density using an eigenfunction decomposition, an approach used for diffusions in the litterature \citep{Davydov,ItoMacKean,MacKean}. For the CIR process, it is established in \cite{linetskya,linetskyb} that the same type of decomposition holds. We now recall the result of \cite{linetsky04} that provides a series expansion for the density $f_{T_{x\rightarrow y}}$ of the first hitting time of $y$ starting from $x$. 

\begin{proposition}{\cite{linetsky04}}
\begin{enumerate}
\item[i)] For $0<x<y \in I$, and $t>0$ we have
\begin{equation}\label{eq:time_density}
f_{T_{x\rightarrow y}}(t) = \sum_{n=1}^{+\infty} o_n \lambda_n  e^{-\lambda_n t}\,,
\end{equation}
with uniform convergence on $[t_0,+\infty)$, $t_0>0$, and $(\lambda_n)_n$ a strictly positive and strictly increasing sequence with $\lambda_n$ growing to $+\infty$ as $n$ goes to infinity. More precisely, we have that 
\begin{equation}\label{eq:lambda}
\lambda_n = (r-d) s_n\,,
\end{equation}
with $(s_n)_n$ the strictly decreasing sequence of strictly negative roots of $\Phi(\cdot;c/a;\overline{y})=0$, with $\Phi(w_1;w_2;w_3)$ denotes the Kummer confluent hypergeometric function. The sequence $(o_n)_n$ is defined by 
\begin{equation}\label{eq:on}
o_n = - \frac{\Phi(s_n;c/a;\overline{x})}{s_n \partial_s(\Phi(s_n;c/a;\overline{y}))}\,,
\end{equation}
for $\overline{y} :=- \frac{r-d}{a}y$ and $\overline{x} :=- \frac{r-d}{a}x$.
\item[ii)] Moreover, the following asymptotics hold: 
\begin{equation}\label{eq:lambdan_large}
\lambda_n \underset{n\rightarrow +\infty}{\sim} \frac{(d-r)\pi^2}{4 \overline{y}} \left(n + \frac{c}{2a} - \frac{3}{4} \right)^2 - \frac{(r-d)c}{2a}\,,
\end{equation}
as well as 
\begin{equation}\label{eq:cn_large}
\begin{split}
o_n \underset{n\rightarrow +\infty}{\sim} \frac{(-1)^{n+1} 2\pi (n+ c/(2a) - 3/4)}{\pi^2 (n + c/(2a) - 3/4)^2 - \frac{2c}{a}\overline{y} } \times  e^{\frac{1}{2}(\overline{x}-\overline{y})} \left(\frac{\overline{x}}{\overline{y}} \right)^{\frac{1}{4} - \frac{c}{2a}} \\
\cos\left(\pi \left(n + \frac{c}{2a} - \frac{3}{4} \right)  \sqrt{\frac{\overline{x}}{\overline{y}}} - \frac{\pi c}{2a} + \frac{\pi}{4}\right)\,.
\end{split}
\end{equation}
\end{enumerate}
\end{proposition}
Therefore, the proposed numerical method requires the computation of the set of negative roots of $\Phi$ to get approximations of the families $\{\lambda_n \}_n$ and $\{o_n \}_n$. The choice of the level of truncation for the approximation of \eqref{eq:time_density} may be made using the following estimate: 
\begin{displaymath}
\left| o_n \lambda_N e^{-\lambda_N t_0} \right| \underset{N\rightarrow +\infty}{\sim} A N e^{-BN^2 t_0}\,,
\end{displaymath}
for 
\begin{displaymath}
A = \frac{2a \pi}{4y} e^{\frac{\overline{x}-\overline{y}}{2}} \left(\frac{x}{y} \right)^{\frac{1}{4} - \frac{c}{2a}}\,, \quad B = \frac{a \pi^2}{4 y}\,.
\end{displaymath}
Linetsky also notices that using \eqref{eq:lambdan_large}-\eqref{eq:cn_large} instead of computing zeros of the Kummer function provides satisfactory results, in particular when $c/a$ is small. For better accuracy, one can also use the exact expressions \eqref{eq:lambda}-\eqref{eq:on} of the first term of the decomposition, then the estimates \eqref{eq:lambdan_large}- \eqref{eq:cn_large} for the others. The following numerical simulation was performed using the asymptotic expansions of $\lambda_n$ and $o_n$ only, even for $n$ small, since it is observed that this approximation does not change qualitatively the profile (see Figure \ref{fig:CIR_hit}). Obviously, in this case, the obtained function is not a probability density, and an inconsistency tends to appear near $t=0$ due to the approximation, but the overall shape is preserved. 

\begin{figure}
\centering
\includegraphics[scale=0.5]{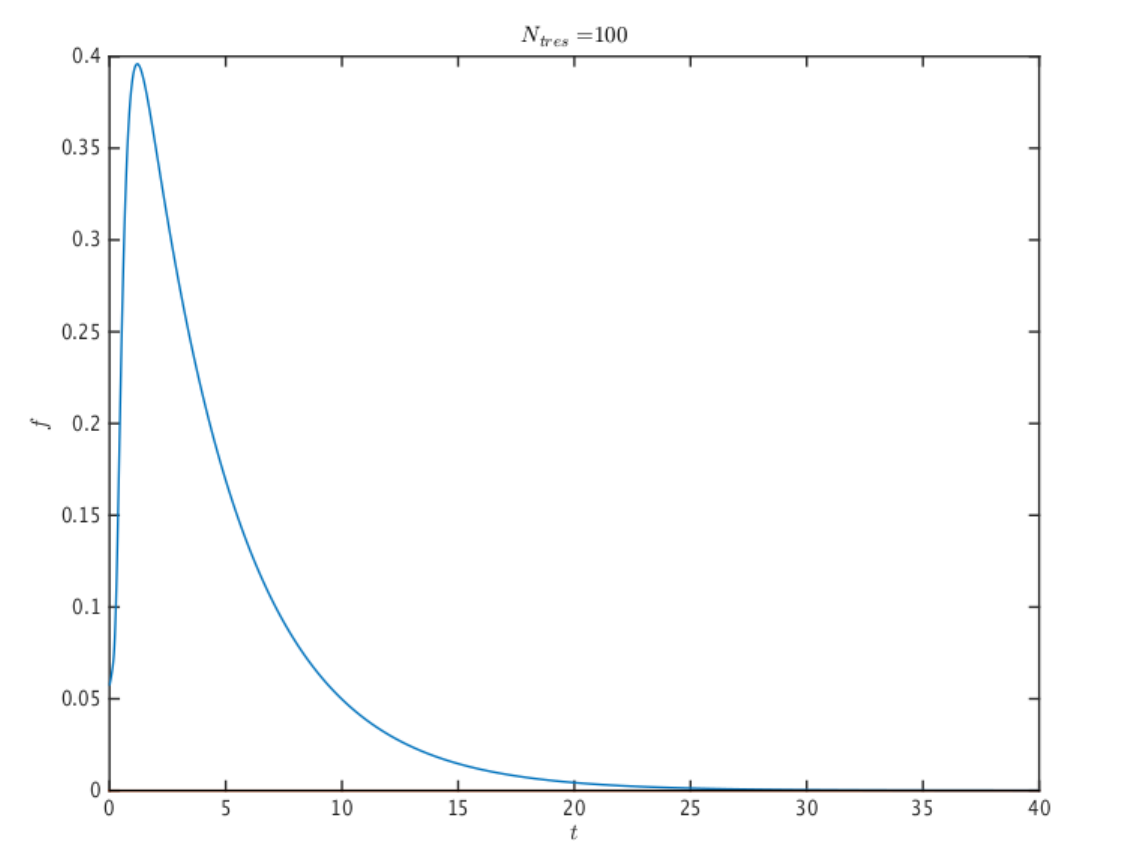}
\caption{Numerical simulation of an approximation of the asymptotic spectral decomposition of \eqref{eq:time_density}, the probability density of the first hitting time of $1$ of a CIR process starting at $0.01$. Parameters: $\Delta t=0.01$, $c=0.45$, $a=0.5$, $r=0.2$ and $d=1$. The sum is truncated at $N_{\text{tres}}=100$.}\label{fig:CIR_hit}
\end{figure}

\subsubsection*{The general case}
Let us now focus on the general case of Equation \eqref{SDE_discret}. As a first approach one can use the 1D comparison principle (see e.g \cite{revuz}) to compare the process with CIR processes. In this work, we follow another method and derive from \eqref{SDE_discret} a Fokker-Planck equation on $p(n,t) := p(n,t|n_0,t_0)$ the probability density of $(N_t)_t$ conditionally to its initial condition. We obtain the following equation: 

\begin{equation*}
 \frac{\partial p(n,t)}{\partial t} = \frac{\partial }{\partial n}  \underbrace{( -b(n) p(n,t) + \frac{1}{2}\frac{\partial }{\partial n}(\sigma^2(n) p(n,t)) )}_{\mathbf{J}(n,t)} \,,
\end{equation*}
where we recall that $b(n) = c + (r(n) - d(n))n$, while $\sigma(n) = \sqrt{2an}$ and $\mathbf{J}(n,t)$ is the associated probability current. The natural boundary conditions are the following:
\begin{eqnarray*}
\mathbf{J}(0,t) &=& 0 \, ,\\
\lim_{n\rightarrow + \infty}  p(n,t) &=&0\, ,\\
p(n,0) &=& \delta_{n=n_0}\, .
\end{eqnarray*}
We are interested in the mean time necessary for the process to reach the value $u$ starting from $n_0 \in (0,u)$, that we denote by $\tau_u(n_0)$.
This question can be adressed by considering the Fokker-Planck equation on $(0,u)$ with $0$ a reflecting and $u$ an absorbing barrier. We show the following proposition. 

\begin{proposition}
The mean time $\tau_u(n_0)$ necessary for the process to reach the value $u$ starting from $n_0 \in (0,u)$ writes
\begin{equation}\label{eq:MeanTimeStop}
\tau_u(n_0) = \frac{1}{a} \int_{n_0}^{u} \int_{0}^{y} \left( \frac{z}{y} \right)^{\frac{c}{a}} z^{-1} e^{\frac{r}{a}(z-y)} \exp\left(\frac{d}{a \alpha } e^{\alpha u} \left(e^{-\alpha z} -  e^{-\alpha y}\right) \right)
\dx z \, \dx y \, .
\end{equation}
\end{proposition}

\begin{proof}
Write $G(n_0,t)$ the probability that a particle starting at $n_0$ is still in $(0,u)$ at time $t$. Then,
\begin{displaymath}
G(n_0,t) = \int_0^{u} p(n,t|n_0,0) \Dx n = \mathbb{P}(\tau_u\geq t)\,.
\end{displaymath}
Since the dynamics is homogeneous in time, we deduce that $p(n,t|n_0,0) = p(n,0|n_0,-t)$ and $n_0 \mapsto p(n,t|n_0,0)$ satisfy the backward Fokker-Planck equation:
\begin{equation*}
\frac{\partial p(n,t|n_0,0)}{\partial t} =   b(n_0) \frac{\partial }{\partial n_0}p(n,t|n_0,0) + \frac{1}{2}\sigma^2(n_0) \frac{\partial }{\partial {n_0}^2}p(n,t|n_0,0)  \,,
\end{equation*}
and $(n_0,t)\mapsto G(n_0,t)$ satisfies
\begin{equation}\label{eq:G}
\frac{\partial G(n_0,t)}{\partial t} =   b(n_0) \frac{\partial }{\partial n_0}G(n_0,t) + \frac{1}{2}\sigma^2(n_0) \frac{\partial }{\partial {n_0}^2}G(n_0,t)  \,.
\end{equation}
The initial and boundary conditions are the following: 
\begin{eqnarray*}
G(n_0,0) &=& \int_{0}^{u} \delta_{n-n_0} \Dx n = \mathds{1}_{[0,u]}(n_0)\,,\\
\frac{\partial }{\partial n_0}G(0,t) &=& 0\, ,\\
G(u,t) &=& 0\,.
\end{eqnarray*}
Take $f\in \mathcal{C}^1(\R,\R_+)$ non-decreasing. Then, classically, $\mathbb{E}[f(\tau_u)] = \int_0^{+\infty} f'(t) \mathbb{P}(\tau_u>t) \Dx t =  \int_0^{+\infty} f'(t) G(n_0,t) \Dx t$. Hence, we get for $k>1$,
\begin{eqnarray*}
\tau_u(n_0) &=& \mathbb{E}[\tau_u] = \int_0^{+\infty}G(n_0,t) \Dx t\,,\\
\tau_u^k(n_0) &=& \mathbb{E}[\tau_u^k] = k \int_0^{+\infty} t^{k-1} G(n_0,t) \Dx t\,.
\end{eqnarray*}
Integration of \eqref{eq:G} in time leads to the following ODEs on the family $(\tau_k)_{k\geq 1}$:
\begin{equation}\label{eq:ODE_T}
\begin{cases}
&b(n_0) \tau_u'(n_0) + \frac{1}{2} \sigma^2(n_0) \tau_u''(n_0) = -1\,,\,\\
&\tau_u'(0) = 0\, , \\
&\tau_u(u)= 0\,,
\end{cases}
\end{equation}
and for $k>1$,
\begin{equation}\label{eq:ODE_Tk}
\begin{cases}
&b(n_0) {\tau^k_u}'(n_0) + \frac{1}{2} \sigma^2 (n_0) {\tau^k_u} ''(n_0) = -k \tau^{k-1}_u (n_0)\,,\\
&\partial_{n_0} \tau^k_u (0) = 0\, , \\
&\tau^k_u (u)= 0\,. 
\end{cases}
\end{equation}
By direct integration, we can solve \eqref{eq:ODE_T}, allowing to solve successively the problems \eqref{eq:ODE_Tk}. Write
\begin{equation*}
\Psi(n_0) = e^{\int_0^{n_0} \frac{2b(n')}{\sigma^2(n')} \Dx n'}\,.
\end{equation*}
Then, we have 

\begin{equation*}
\tau_u(n_0) = 2 \int_{n_0}^{u} \frac{1}{\Psi(y)} \int_{0}^{y} \frac{\Psi(z)}{\sigma^2(z)} \Dx z \, \Dx y\, .
\end{equation*}
In practice, denoting $\epsilon>0$ the lower bound in the integral instead of zero, we find that
\begin{equation*}
\Psi(n_0) = \left( \frac{n_0}{\epsilon} \right)^{\frac{c}{a}} \exp\left(\frac{r}{a}n_0 - \frac{d}{\alpha \gamma a}e^{\alpha u} (1 - e^{-\alpha \gamma n_0}) \right)\,.
\end{equation*}
Explicit computations lead to the result. 
\end{proof}

We perform numerical simulations of $\tau_u(0)$ for different values of the blood flow velocity $u$ (see Figure \ref{fig:temps_atteinte}). Below a shear-velocity treshold, the particle stops very quickly, while it becomes extremely slow above the treshold. The numerical phase plane shows a natural dependency on the adhesion creation rate $c$. The value of $a$, quantifying the noise intensity, has a direct effect on the range of the arrest time, but does not qualitatively change the phase plane.

\begin{figure}%
\captionsetup[subfigure]{justification=centering}
\begin{center}
\subfloat[Mean first stopping time $\tau_u(0)$ as a function of $u$.]{
\includegraphics[width=0.47\textwidth]{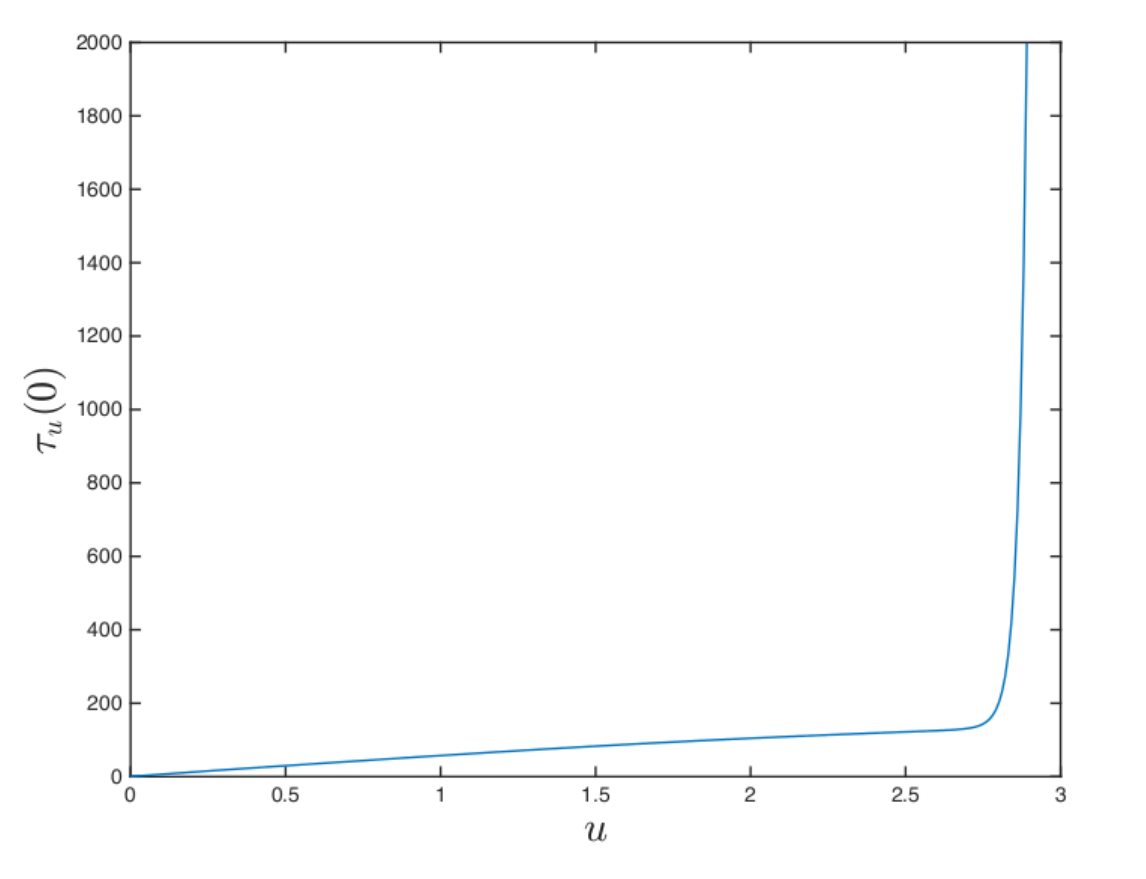} 
}
\subfloat[Phase plane with respect to $u$ and the spontaneous binding rate $c$.]{
\includegraphics[width=0.42 \textwidth]{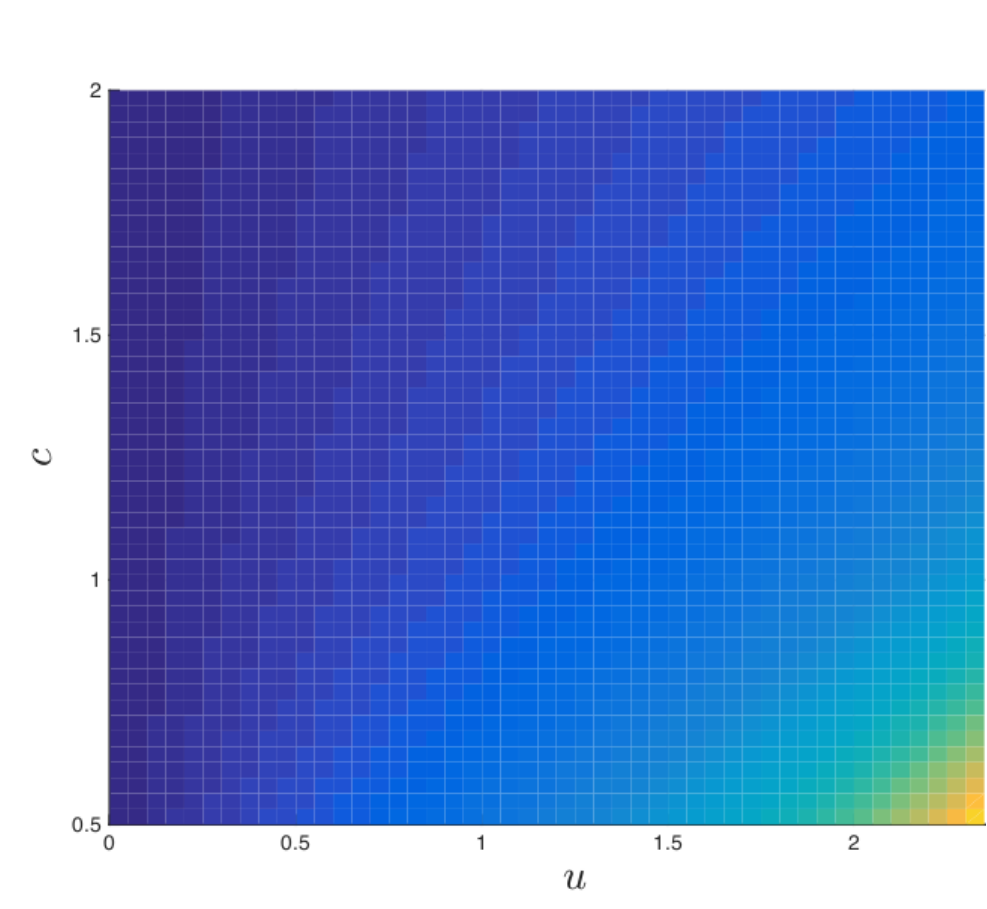}
}
\caption{Numerical simulations of the mean stopping time $\tau_u(0)$ defined by \eqref{eq:MeanTimeStop} as a function of $u$ (left), and phase plane depending on the blood velocity $u$ and the spontaneous binding rate $c$. Parameters: $\alpha=0.8$, $r=0.6$, $d=0.7$, $a=0.1$. }\label{fig:temps_atteinte}
\end{center}
\end{figure}

\section{Discussion}\label{discussion}

In this work, we have presented a discrete model of cell adhesion on a vessel wall. The model is based on the stochastic formation of both weak bonds between the cell and the wall, and stronger ones arising by self-reinforcement. This phenomenon is modelled by a stochastic birth-and-death-type process. The cell velocity exerts a feedback on the breaking rate of bonds: the faster the cell, the shorter-lived the bonds.\par 
Our purpose is to explore the events that follow the first cell contact with the wall and that may determine whether the cell stops and adheres or goes back into the blood stream. More precisely, we are interested in reproducing the bistable behaviour observed experimentally between cell rolling on the wall and its stationary adhesion. Moreover, we look for an expression of the 
cell stopping time, which is equivalently the first hitting time of a treshold value for the density of bonds. To achieve these goals, we perform some scaling limits to derive continuous deterministic and stochastic models that allow for deeper theoretical analysis. 

\subsection{Implication of our work on the understanding of the arrest and extravasation of circulating cells}

In vivo, the firm adhesion of circulating cells on the wall is a first step towards its extravasation out of the vessel, which has different consequences depending on the cell type. Extravasation of  Circulating Tumor Cells (CTCs) takes part in tumor invasion since it allows the formation of secondary tumors \citep{follain2018hemodynamic}. During the immune response, leukocytes carried by the bloodstream firmly adhere to the vessel wall at inflammed sites. Their extravasation then allows them to pursue their immune function. Let us also mention that cell adhesion to vessel walls is involved in the study of drug delivery systems. 

Overall, cell adhesion to vessel walls is a phenomenon showing major applications in biology and medecine. It is now clear that understanding the processes involved in the determination of the location of cell arrest is of prime importance. This justifies the development of new mathematical models able to explain the experimental observations. Our model focuses on the cell behaviour within the bloodstream, and is consequently no longer valid after the cell velocity has reached zero. A different model is then required to capture the fundamental processes involved in extravasation. 

\subsection{The linear mean ODE \eqref{eq:meanODE} does not agree with known observations}

Our first finding concerns the mean linear Ordinary Differential Equation \eqref{eq:meanODE} in which the cell velocity has no effect on the adhesion dynamics: ligand binding and dissociation rates are constant. As expected, in such a case, the stationary mean number of bonds is independent of the shear rate $u$. Such a result does not agree with in vitro and in vivo experiments, \citep{follain2018hemodynamic,Goetz}, which show that cells preferably stop in regions with low shear rate. This shows that a nonlinear model is required to capture this phenomenon.

\subsection{The nonlinear model qualitatively agrees with known observations: hemodynamic forces affect the adhesion dynamics of circulating cells \citep{follain2018hemodynamic}}

In the nonlinear case, we perform a renormalization procedure based on the biological observation that the orders of magnitude between the adhesion dynamics and the cell motion are different. This allows us to rigorously build continuous models, either deterministic or stochastic, of the dynamics of the density of bonds (or equivalently of the cell velocity). Then, the analysis shows that these nonlinear models predict the cell adhesion bistability as a result of the competition taking place in the cell-wall contact area between bond formation and rupture. 

\paragraph{The deterministic limiting model provides a parameter space for cell adhesion bistability.}
The deterministic limiting model writes as a nonlinear ODE, for which we give the stationary states and their stability. This allows us to outline the existence of a dichotomy behaviour in which the stability of the cell arrest is explicitly related to the balance between the blood flow velocity and the adhesion dynamics. It leads to the identification of a shear stress value that separates the cell arrest from a moving state, in agreement with experimental observations \citep{follain2018hemodynamic}. 
 
Our study also provides a parameter space associated with the existence of two stationary states, with a stable release state and an unstable sliding one. Moreover, since the model is no longer valid after the bond density reaches $u$, the cell arrest is an additional «stable» stationary state. Therefore, the model is indeed able to reproduce the bistability observed in experimental works and already reproduced in more elaborate models that take into account the types of proteins involved in the adhesion and the role of hydrodynamic forces in the cell fate. 

In \cite{finger1996adhesion}, the authors investigate the specificity of short-lived L-selectin proteins in the adhesion process. More precisely, they put to light a shear-flow treshold above which L-selectin-based binding occurs. The authors emphasize that the first step in cell adhesion may require the recruitment of these short-lived proteins to sufficiently slow the cell down to promote the formation of longer-lived types of bonds. Let us note that our model can explain this observation. Indeed, assuming that the spontaneous formation of bonds is related to L-selectins, the observed treshold is described using a bond formation rate of the form $c(u,v) = c \mathds{1}_{[0,u_*]}(u)\mathds{1}_{[\underline{v},+\infty)}(v)$, for a given $\underline{v}>0$. Reproduction then accounts for the integrins dynamics. Moreover, our assumption of velocity-dependent dissociation rate is consistent with observations of \cite{finger1996adhesion}, since L-selectin bonds form when the cell velocity is high and, consequently, so is the dissociation rate. Integrin-based bonds form for a smaller cell velocity, leading to longer lifetimes. Initially, one can assume that $n(0)=0$, $c\neq 0$, $r<<d$, and that $u> \bar{U}_{\alpha}$. Table \ref{table:limdet2} then shows that in the long-time limit, the density of bonds converges towards a stable sliding state. Then, if the cell velocity gets low enough, the cell dynamics is ruled by Table \ref{table:limdet1} with $r>d$, where the bistable behaviour between cell release and cell arrest is given. Our model is therefore able to succesfully explain the two main steps involved in the experimental observations.

More recent works have provided mechanical models taking into account the hydrodynamics forces involved in cell adhesion. In \cite{reboux2007bond}, the cell is a rolling cylinder, and interestingly the model captures both sliding and rolling situations. The model shows a bistable behaviour between cell rolling (or weak resistance to sliding) and its release in the blood flow. In \cite{efremov2011bistability,li2018rolling}, the cell is either a rectangle with a given zone for bond rupture, or a 2D circular cell in adhesive interaction with an elastic substrate. Both models predict bistability between the cell arrest and a near-release rolling state for intermediate shear rates. Furthermore, plots of the stationary velocity depending on the blood flow velocity show that, at stationary state, the cell adheres for a low shear rate and is in a rolling state above some threshold. The explanation relies on the competition between formation and rupture of bonds, that vary with the proteins involved in the adhesion. 
%

The model we build in this paper is simpler, since it does not describe the cell geometry, so that the hydrodynamic description is minimal. As a consequence, cell rolling cannot be distinguished from sliding, and we are not able to provide a plot relating the blood and cell velocities. However, remarkably, the model still reproduces a bistable behaviour, suggesting that cell adhesion dynamics is crucial for this feature.

\paragraph{The stochastic limiting model provides an expression for the cell mean arrest time and shows the effect of fluctuations on the dynamics.}

The stochastic limiting model writes as a diffusive Stochastic Differential Equation on the adhesion density, that carries a nontrivial noise term. It arises in a regime where the fluctuations in the local reinforcement of adhesions have nontrivial effects on the cell dynamics. When there is no effect of the cell velocity on its adhesion activity, the model writes as a linear CIR process, that already carries a dichotomic behaviour, and for which some first hitting time properties are known. We use a spectral method to perform numerical simulations of the corresponding probability density. Finally, in the full model describing the interaction between the cell velocity and its adhesion activity, we use a Fokker-Planck approach to derive an integral formulation of the mean arrest time of a cell. Numerical simulations of this quantity further confirm the dichotomy arising between the cell arrest and its unstopped displacement in the bloodstream, and show that the range of the arrest times is mostly sensitive to the stochastic fluctuations intensity, in agreement with experimental observations.

\subsection{Perspectives}
The modelling approach we developped gathers key components involved in the behaviour of cells circulating in blood vessels. It allows the study how the blood flow affects the cells ability to initiate and maintain an adhesive interaction with endothelial cells. These phenomena mediate either the stable cell adhesion to the wall, its rolling or sliding, or its release in the blood flow. 

Further improvements would consist in extending the modelling framework in several directions. First, the model could take into account time-dependent rates, in order to consider a variable blood-flow velocity, illustrating the effect of the heart cycle on the adhesion dynamics. Moreover, it is natural to extend the model to a 2D setting, where the vessel wall is a surface and the cell geometry is taken into account, as well as the spatial repartition of adhesive proteins at the cell surface. This situation could be handled by adding a spatial structure to the population of bonds, following the framework of \cite{fournier_microscopic_2004,champagnat_invasion_2007} and further works on measure-valued stochastic processes. This framework would allow the precise description of hydrodynamic forces exerted on the cell, so that both rolling and sliding phenomena could occur. 

\noindent Finally, further works should focus on the comparison with experimental measures to to shed light on cell arrest conditions. In particular, in the work of \cite{follain2018hemodynamic}, the authors study Circulating Tumor Cells (CTCs) in vivo, and keep track of the cell arrest sites with respect to a tuned hemodynamic flow velocity. They show that cell arrest in blood vessels occurs at sites with permissive flow profiles. Since CTCs are large, fairly rigid cells, our point-particle cells assumption is particularly relevant in this context. In addition, the density of adhesion molecules on the vessel wall is experimentally controlled, so that each feature of the model is characterized by experimental data. In this perspective, confronting the model to these data is of great interest, and could be supplemented with the theoretical study of the distribution of arrest positions.

\begin{acknowledgements}
The authors are very grateful to V.C. Tran and R. Voituriez for very helpful discussions and suggestions. They also wish to thank the reviewers for their relevant suggestions, and F. Gidel for proofreading the manuscript.
\end{acknowledgements}

\newpage
\section*{}\label{ref}
\bibliographystyle{spbasic}      
\bibliography{ligands}   

\begin{appendices}
\section{Mathematical properties of the rescaled process}\label{annex1}
In this part, we study the rescaled process $(Z_t^K)_t$ for a constant $K$, in a general framework that includes the situations studied in this paper. More precisely, we consider rates that satisfy the following hypothesis:
 
\begin{hypo}\label{hypotheses_taux_renorm1D}
For all $z \in \R_+$,
\noindent
\begin{center}
\begin{tabular}{cc}
$0\leq c^K (u) \leq Kc(u) \leq Kc $, & $0\leq r^K \leq r + K a$,\\
&\\
\multicolumn{2}{c}{$0\leq d^K (z) \leq d(z) + K a \leq d e^{\alpha u} + Ka$,}
\end{tabular}
\end{center}
\noindent
where $a>0$. In addition the disassembly rate $z \mapsto d^K (z)$ is continuous. The total formation and dissociation rates write
\begin{displaymath}
\begin{array}{cc}
\lambda^K(z) = c^{K}(u) + r^K z \text{ and } & \mu^K(z) = d^K(z) z\,.
\end{array}
\end{displaymath}
\end{hypo} 
By construction, $(Z_t^K)_{t\geq 0}$ is also a Markov process, and for $\Phi:\R_+ \rightarrow \mathbb{R}$ measurable bounded, its infinitesimal generator writes 
\begin{equation}\label{generateur_renorm1D}
L^K\Phi(Z) =  \lambda^K(KZ) \left[ \Phi(Z + \frac{1}{K} ) - \Phi(Z) \right] + \mu^K(KZ)  \left[\Phi(Z - \frac{1}{K}) - \Phi(Z) \right]\, .
\end{equation}
We show the following Proposition.
\begin{proposition}[Moment and martingale properties]\label{existence_renorm1D}
Under Hypothesis \ref{hypotheses_taux_renorm1D}, if there exists $p \geq 2$ such that $\mathbb{E}\left[(Z_0^K)^p\right] <+\infty$, then
\begin{enumerate}
\item  $\forall \; T>0,$
\begin{equation*}
\mathbb{E}\left[\sup_{t\in [0,T]} (Z^K_t)^p\right]
<+\infty,
\end{equation*}
\item for all measurable $ \Phi : \R_+ \rightarrow \mathbb{R} $ for which there exists $C$ such that $\forall z \in \R_+, \; |\Phi(z)| + |L^K\Phi(z)| \leq C (1+z^p) $, 
\begin{equation}\label{Mart_phi1D}
\Phi(Z^K_t) - \Phi(Z^K_0) - \int_{0}^t L^K\Phi(Z^K_s) \Dx s
\end{equation}
is a c\`adl\`ag $(\mathcal{F}_t)_{t\geq 0}$-martingale starting from $0$.
\item The process defined by
\begin{equation}\label{martingale_K1D}
\begin{split}
M^{K}_t = Z^K_t - Z^K_0  - \int_{0}^t  \frac{1}{K} c^K(KZ^K_s)  + \left( r^K(KZ^K_s) -d^K(KZ^K_s) \right)   Z_s^K \Dx s 
\end{split}
\end{equation}
is a c\`adl\`ag square-integrable martingale starting from $0$ and of quadratic variation
\begin{equation}\label{quadratic_K1D}
\begin{split}
\left< M^{K}\right>_t  = \frac{1}{K} \int_0^t \left\{ \frac{1}{K} c^K(KZ^K_s)  + (r^K(KZ^K_s) + d^K(KZ^K_s))Z_s^K \right\} \Dx s\, .
\end{split}
\end{equation}
\end{enumerate}
\end{proposition}

\begin{proof}
Although the proof is similar to Proposition 2.7 of \citep{Bansaye2015} we recall it here for clarity. In the following $C$ will denote a positive constant which value can change from line to line. Let $(\Omega,\mathcal{F},\mathbb{P})$ be a probability space, $Z_0$ an integer-valued random variable, and 
$M(\Dx s,\dx w)$ an independent Poisson Point Measure on $\mathbb{R}^2_+ $, of intensity measure $\Dx s \dx w$. Denote by $(\mathcal{F}_t)_{t\geq 0}$ the canonical filtration generated by these objects. Then, we define the $(\mathcal{F}_t)_{t\geq 0}$-adapted c\`adl\`ag  process $(Z^K_t)_{t\geq 0}$ as the solution of the following SDE: $\forall t\geq 0$,
\begin{equation*}
Z^K_t = Z^K_0 + \int_{0}^t \int_{\mathbb{R}_+}  \left( \mathds{1}_{0\leq w \leq \lambda^K(Z^K_{s^-})} - \mathds{1}_{\lambda(Z^K_{s^-})< w \leq \lambda^K(Z^K_{s^-}) +\mu^K(Z^K_{s^-})} \right) M(\Dx s,\dx w) \, .
\end{equation*}
This representation is classical (see e.g \cite{fournier_microscopic_2004,champagnat_invasion_2007}). The Poisson jumps related to the measure are accepted or rejected thanks to the indicator functions. The variable $w$ is then used as an acceptance parameter in order to obtain the desired rates for each event. Now, writing that $\mathbb{P}-a.s$, for a positive and measurable test function $\Phi$, 
\begin{equation*}
\begin{aligned}
\Phi(Z^K_t) &= \Phi(Z^K_0)  + \int_0^t \int_{\R_+} \left[ (\Phi(Z^K_{s^-} +1) - \Phi(Z^K_{s^-} ) \mathds{1}_{0\leq w \leq \lambda^K(Z^K_{s^-})} \right.\\
&\left. +  (\Phi(Z^K_{s^-} -1) - \Phi(Z^K_{s^-} ) \mathds{1}_{ \lambda^K(Z^K_{s^-})\leq w\leq  \lambda^K(Z^K_{s^-}) + \mu^K(Z^K_{s^-})} \right] M(\Dx s, \Dx w) \,,
\end{aligned}
\end{equation*}
so that for $\Phi(Z^K_t) = \left(Z^K_t\right)^p$ and neglecting the negative death term, we obtain
\begin{equation*}
\left(Z^K_{t}\right)^p \leq   \left(Z^K_0\right)^p + \int_{0}^{t} \int_{\mathbb{R}_+} 
\left((Z^K_{s^-}+1)^p - \left(Z^K_{s^-}\right)^p \right)
\mathds{1}_{0\leq w \leq \lambda^K(Z^K_{s^-})}  M(\Dx s, \Dx w) \,.\\
\end{equation*}
Taking expectation, and using that for $p\in \N$, $(1+z)^p - z^p \leq C(p)(1+z^{p-1})$, we obtain that 
\begin{eqnarray*}
\mathbb{E}\left[\sup_{t\in [0,T\wedge \tau_u]}\left(Z^K_{t}\right)^p\right] &\leq & \mathbb{E}[\left(Z^K_0\right)^p] + C(p) \mathbb{E}\left[\int_{0}^{T\wedge \tau_u} \left(1
+Z^K_{t^-} \right) \left(1 + \left(Z^K_{t^-}\right)^{p-1} \right) \Dx t \right] \\
&\leq &\mathbb{E}[\left(Z^K_0\right)^p] + C(p) \left(T +  \int_{0}^{T} \mathbb{E}\left[\sup_{u\in [0,t\wedge \tau_u]}\left(Z^K_{u^-}\right)^{p}\right] \Dx t\right)\,.
\end{eqnarray*}
We conclude with the Gronwall Lemma together with the assumption on the initial condition. The second item of Proposition \ref{existence_renorm1D} is a classical property of Markov processes, and \eqref{martingale_K1D} is obtained as a consequence. Finally, to obtain the quadratic variation \eqref{quadratic_K1D}, one has to write the semi-martingale decomposition of $\left(Z_t^K\right)^2$, one the one hand by applying \eqref{Mart_phi1D} to $\Phi(z)=z^2$, and on the other hand by applying \eqref{Mart_phi1D} to $\Phi(z)=z$ before using the It\={o} formula. The conclusion follows from the uniqueness of the decomposition. 
\end{proof}

\section{Proof of Theorem \ref{lim_deterministe_1D}}\label{app:det_conv}

The proof is similar to the ones of \citep{joffe_weak_1986,ethier2009markov}, and is based on a compactness-uniqueness argument. First, note that since the rates under study have $K$-independent upper bounds, reproducing the proof detailed in Appendix \ref{annex1} lead to
\begin{equation*}
\forall \; T>0,\; \sup_{K>0} \mathbb{E}\left[\sup_{t\in [0,T \wedge \tau_u]}(Z^K_t)^p\right] <+\infty\,.
\end{equation*}
The uniform tightness of the sequence of laws $(Q^K)_K$ of $(Z^K)_K$ follows. Then, from the Prokhorov theorem we deduce the relative compactness of the family of laws $(Q^{K})_K$ on $\mathbb{D}([0,T], \R_+)$. Consider now a convergent subsequence of limit $Q$, and a corresponding sequence of processes converging in distribution to some $n \in \mathbb{D}([0,T], \R_+)$ of law $Q$. Our aim is thus to identify this limit. Firstly, since the jumps of $(Z_t^{K})_t$ are of the form $1/K$, we know that any process of law $Q$ is almost surely strongly continuous. Now, for $t\leq T\wedge \tau_u$, denote 
\begin{eqnarray*}
\Psi_t (n)& :=& n_t - n_0 - 
\int_0^t  c(u)   +  (r- d(n_s)) n_s \Dx s \, .
\end{eqnarray*}
We can easily prove that for all $t\leq T$, $\mathbb{E}_Q\left[ \left| \Psi_t(n) \right|\right]=0$. Finally, the convergence follows from the uniqueness of solutions to \eqref{limite1_eq1D} in $\mathcal{C}([0,T],\R_+)$, which comes from the Lipschitz-continuity of the disassembly rate.

\section{Proof of Theorem \ref{theo_stoch1D}}\label{annex:stoch}
\subsection{Uniform estimates}
We first prove two propositions that provide uniform estimates on the process. They will be used to show the tightness of any sequence of laws associated with $(Z_t^K)_K$. 

\begin{proposition}\label{moments_stoch1D}
Under the assumptions of Theorem \ref{theo_stoch1D}, if $$\sup_{K>0} \mathbb{E}[(Z_0^K)^2] < + \infty,$$ then for $T< +\infty$, 
\begin{equation*}
\sup_K \sup_{t\in [0,T]} \mathbb{E}[(Z_t^K)^2] <+\infty.
\end{equation*}
\end{proposition}

\begin{proof}
First, the infinitesimal generator \eqref{generateur_renorm1D} associated with $\Phi(x) = x^2 $ writes
\begin{eqnarray*}
L^K\Phi(Z^K_s)&\leq & \left(c +r Z^K_s\right) \left(2 Z^K_s + \frac{1}{K}\right)  + 2 a Z^K_s \leq  C \left(1 + Z^K_s + (Z^K_s)^2 \right)\,.
\end{eqnarray*}
Therefore, we obtain that
\begin{eqnarray*}
\mathbb{E}[(Z_t^K)^2] &\leq & \mathbb{E}[(Z_0^K)^2] +  C \left( t + \int_0^t \E[Z^K_s]  +   \E[(Z^K_s)^2] \Dx s\right)\,,
\end{eqnarray*} 
and since the Gronwall Lemma yields $\E[Z^K_s]\leq C(1+\mathbb{E}[(Z_t^K)^2] )$, we deduce that there exists $C(T)$ such that $\mathbb{E}[(X_t^K)^2] \leq C(T)$, and the Proposition is proved. 
\end{proof}

\begin{proposition}\label{moments_stoch21D}
Under the assumptions of Theorem \ref{theo_stoch1D}, if $\sup_{K>0} \mathbb{E}[(Z_0^K)^2] < + \infty$, then for $T< +\infty$, 
\begin{equation*}
\sup_K \E\left[ \sup_{t\in [0,T]} Z_t^K\right] <+\infty\,.
\end{equation*}
\end{proposition}

\begin{proof}
We first deduce from \eqref{martingale_K1D} that 
\begin{equation}\label{eq:prop2det}
\sup_{t\in [0,T]} Z^K_t \leq  \sup_{t\in [0,T]} |M^{K}_t| + Z^K_0  +  cT + r\int_{0}^t Z^K_s \Dx s\,.
\end{equation}
We want to take the expectation in this inequality. For that purpose, we first use the Burkholder-Davis-Gundy inequality to write that
\begin{displaymath}
\E\left[\sup_{t\in [0,T]} |M^{K}_t|\right]^2 \leq \E\left[\sup_{t\in [0,T]} |M^{K}_t|^2\right] \leq 4 \E\left[|M^{K}_{T}|^2\right] = 4 \E\left[\left<M^{K}\right>_T\right]\,.
\end{displaymath}
Now, since $\E[Z^K_0]<+ \infty$, we obtain from \eqref{eq:prop2det} that
\begin{eqnarray*}
\E\left[\sup_{t\in [0,T]} Z^K_t\right] &\leq & 2\E\left[\left<M^{K}\right>_{T}\right]^{1/2} + C(T) + r\E\left[\int_{0}^t Z^K_s \Dx s\right].
\end{eqnarray*}
We then use \eqref{quadratic_K1D} to get 
\begin{eqnarray*}
\E[\left< M^{K}\right>_{T}]  &\leq & cT+  (r + de^{\alpha u}+ 2 a)  \int_0^t \E[Z^K_s] \Dx s \leq  C(T)
\end{eqnarray*}
thanks to Proposition \ref{moments_stoch1D}. We conclude using the Gronwall lemma.
\end{proof}

\subsection{Proof of convergence}
The proof follows the same outline as in the deterministic case. First, we prove similarly that the sequence of laws $(Q^K)_K$ of the processes $(Z^K)_K$ is uniformly tight in $\mathcal{L}(\mathbb{D}([0,T],\R_+))$. Indeed, denote $(A_t^{K})_{t\geq 0}$ the finite variation process associated to $(Z_t^K)_t$. The Aldous and Rebolledo criterion \citep{1989aldous} states that we need to prove that for all $T>0$ the following inequalities hold true:
\begin{enumerate}
\item[a)] $ \displaystyle \sup_{K >0}\mathbb{E}\left[ \sup_{t\in [0,T]} \left| Z_t^K \right|\right] <+ \infty$.
\item[b)] $\forall \varepsilon >0, \; \forall \eta >0,\; \exists \delta >0,\; K_0 \in \mathbb{N}^*$ such that  for all sequence $(\sigma_K,\tau_K)_K$  of stopping times with $\sigma_K \leq \tau_K \leq T,$
\begin{enumerate}
\item[(i)] \begin{displaymath}
\sup_{K\geq K_0} \mathbb{P}\left(\left| <M^{K}>_{\tau_K} - <M^{K}>_{\sigma_K} \right|\geq \eta,\; \tau_K \leq \sigma_K + \delta \right) \leq \varepsilon,
\end{displaymath}
\item[(ii)] \begin{displaymath}
\sup_{K\geq K_0} \mathbb{P}\left(\left| A^{K}_{\tau_K} - A^{K}_{\sigma_K} \right|\geq \eta,\; \tau_K \leq \sigma_K + \delta \right) \leq \varepsilon\,.
\end{displaymath}
\end{enumerate}
\end{enumerate}
These estimates can easily by obtained by direct computations, using Proposition \ref{moments_stoch21D} and the Markov inequality. Now, we aim at identifying the limiting values. For $Y=(Y_t)_{t\geq 0} \in \mathbb{D}([0,T],\R_+)$, let us define
\begin{equation}\label{proc_limite_1D}
\widetilde{M}_t(Y) := Y_t  - Y_0 -  \int_{0}^t c(u)  +  \left( r - d(Y_s) \right)Y_s  \Dx s\,.
\end{equation}
We need to show that $\widetilde{M}_t(N)$ is a twice-integrable continuous martingale with quadratic variation process defined by 
\begin{equation}\label{proc_quad_1D}
\left< \widetilde{M} \right>_t = 2 a \int_0^t  Y_s \Dx s.
\end{equation}
First, we show that $\widetilde{M}(N)$ is a martingale. Take $0 \leq s_1 <...<s_n<s<t$, and $\Phi_1,...,\Phi_n$ continuous bounded functions from $\R$ to $\mathbb{R}$. Define $\Psi$ on $ \mathbb{D}([0,T], \R)$ by
\begin{equation*}
\Psi(Y) = \Phi_1(Y_{s_1})...\Phi_n(Y_{s_n})\left[ Y_t - Y_s - \int_s^t  c(u) +  \left( r - d(Y_s) \right)Y_s   \Dx u \right].
\end{equation*}
As for the proof of Theorem \ref{lim_deterministe_1D}, we can show that $\E[\Psi(N)] = 0$. The main novelty in this proof consists in showing that the bracket of $\widetilde{M}$ is given by \eqref{proc_quad_1D}. We proceed in two steps. 
\begin{enumerate}
\item First, we consider $K$-dependent semimartingale obtained from \eqref{Mart_phi1D} with $\Phi(Z^K)~=~(Z^K)^2$, that is related to the infinitesimal generator given by
\begin{equation*}
\begin{aligned}
L^K\Phi(Z^K_s)& = 2 Z^K_s \left(c(u) + (r - d(Z^K_s)) Z^K_s \right)   \\
& + \frac{1}{K} \left(c(u) + (r + d(Z^K_s) + 2K a ) Z^K_s \right) \,.
\end{aligned}
\end{equation*}
We can show that, at the limit, we obtain the following martingale:
\begin{equation*}
\stackrel{\sim}{N}_t = (N_t)^2 - (N_0)^2 - \int_0^t  2 N_s  \left(c(u)  + (r - d(N_s)) N_s \right)  + 2 a  N_s  \Dx s \,.
\end{equation*}
\item Then, the It\=o  formula applied to \eqref{proc_limite_1D} show that
\begin{equation*}
N_t^2 - N_0^2 - <\widetilde{M}>_t -\int_0^t  2N_s\left(c(u) + (r - d(N_s)) N_s \right) \Dx s
\end{equation*}
is a martingale. We conclude by the uniqueness of the semimartingale decomposition of $N_t^2$.
\end{enumerate}
The equivalence between the martingale problem \eqref{proc_limite_1D}-\eqref{proc_quad_1D}
and the SDE \eqref{SDE_discret} follows from the classical martingale identification
\begin{equation*}
\widetilde{M}_t = \int_0^t \sqrt{2 a  N_s } \Dx B_s\,,
\end{equation*}
see for example \cite{karatzas}. Finally, the pathwise uniqueness of the solution to \eqref{SDE_discret} is classical in dimension 1 since the drift term is Lipschitz-continuous and the diffusion coefficient is $1/2$-H\"{o}lder (see e.g \cite{ikeda}).

\begin{remark}
The solution is strong and has the strong Markov property.
\end{remark}

\section{The CIR process}\label{annex:CIR}
In this part, we give some classical results about the CIR process. For more details we refer to \cite{GoingJaeschke1999,StephenShreve}.
\subsection*{Reaching zero}
First, it is known that for$c>0$ and a positive initial state, $\{N=0\}$ is a reflective barrier for the process \eqref{eq:CIR}. The properties of the probability $\P^0$ to hit zero are gathered in the following table.
\begin{center}
\begin{tabular}{c|c|c}
\multicolumn{2}{c|}{$c<a$} & \multirow{2}{*}{$c \geq a$} \\
$r-d \leq 0$ & $r-d >0$ & \\
\hline
&&\\
$\P^0 =1$ & $\P^0 \in (0,1)$ & $\P^0 =0$
\end{tabular}
\end{center}
These results can be intuitively understood when $r-d\leq 0$, using the correspondance between CIR processes and Orstein-Uhlenbeck processes. Indeed, consider $D$ such processes $(X^1,\,\cdots ,\, X^D)$ such that $\forall 1\leq i \leq D$, 
\begin{displaymath}
\Dx X^i_t = -\frac{1}{2} \beta X^i_t \Dx t + \frac{1}{2} \sigma \Dx B^i_t\,,
\end{displaymath}
with $(B^i)_i$ independent Brownian motions and $\beta >0$. Each process follows a stochastic dynamics that is drifted to zero. The It\=o formula allows the derivation of the SDE satisfied by $R = (X^1)^2 + \cdots + (X^D)^2$, the squared euclidean norm of $(X^1,\,\cdots ,\, X^D)$. We obtain
 \begin{displaymath}
\Dx R_t = \left(\frac{\sigma^2 D}{4} - \beta R_t \right)\Dx t  + \sigma \sqrt{R(t)} \Dx B_t\,,
\end{displaymath}
with $B$ a Brownian motion. The CIR Equation \eqref{eq:CIR} is therefore obtained for $r-d =-\beta \leq 0$, $\sigma^2=2a$ and $D = 4c / \sigma^2 >0$. As a consequence, when $4c / \sigma^2$ is an integer, the CIR process writes as the squared norm of $D$ Ornstein-Uhlenbeck processes. Therefore, using the properties of the Brownian motin, the CIR process almost surely hits zero infinitely many times when $D=1$, while it has a null probability of reaching zero when $D\geq 2$. As a remark, note also that the CIR process can be rigorously related to the Squared Radial Ornstein-Uhlenbeck process.
\subsection*{Distribution}
When $\delta := \frac{2c}{a}\in \N$ and $r-d<0$, the CIR process writes as the squared norm of a $\delta$-dimensional Ornstein-Uhlenbeck process. Consequently, we have that
 $$N_t|n_0 = \frac{(1-e^{(r-d)t})2a }{2(d-r)} Y_t,$$ 
 where $Y_t$ follows a non-central Chi-square distribution with $\delta$ degrees of freedom, and a non-centrality parameter $\xi_t = n_0 \frac{2(d-r)}{(1-e^{(r-d)t})a}e^{(r-d)t}$. Moreover, the mean solution and its variance can be computed from \eqref{eq:CIR}, leading to
\begin{eqnarray*}
\E[N_t| n_0] &=& n_0e^{(r-d)t} + \frac{c}{d-r} \left(1 - e^{(r-d)t} \right)\,,\\
Var(N_t|n_0) &=& n_0 \frac{2a}{d-r} \left(e^{(r-d)t} - e^{2(r-d)t} \right) + \frac{ac}{ (d-r)^2} \left(1 - e^{(r-d)t} \right)^2\, .
\end{eqnarray*}
For $n\geq 0$, $n_0>0$ and $\kappa_t = \frac{d-r}{(1-e^{(r-d)t})a}$, the probability density $p_{n_0}$ of the CIR process writes
\begin{equation*}
p_{n_0}(n;k,\xi_t) = \kappa_t \left(\frac{n}{n_0e^{(r-d)t}} \right)^{\frac{c}{a}-1} e^{-\kappa_t \left[n_0 e^{(r-d)t} + n \right]}
I_{\frac{c}{a}-1}\left(2\kappa_t \sqrt{n_0 n e^{(r-d)t}}\right)\,,
\end{equation*}
where 
\begin{equation*}
I_\alpha (x) := \Sigma _{m=0}^{\infty}\frac{1}{m! \Gamma (m+\alpha +1)}\left( \frac{x}{2}\right) ^{2m+\alpha }
\end{equation*}
is the modified Bessel function of the first kind, with $\Gamma $ the Gamma function defined by $\Gamma(t)= \int_0^\infty x^{t-1} e^{-x} \Dx x$. A stationary distribution exists if and only if $r-d<0$ \citep{ikeda}. In this case, following a Fokker-Planck approach, one can check that the stationary density writes 
\begin{equation*}
\begin{aligned}
p_{\infty}(n) &= \left(\frac{d-r}{a} \right)^{\frac{c}{a}} \frac{1}{\Gamma\left(\frac{c}{a}\right)} n^{\frac{c}{a} -1} e^{\frac{r-d}{a}n}\,,\\
&=  \mathcal{N} e^{-\phi(n)}\,,
\end{aligned}
\end{equation*}
with $\phi(n) = \left(1-\frac{c}{a}\right)\ln(n) + \frac{d-r}{a}n$ the corresponding potential, and $\mathcal{N}$ a normalization constant. 
\end{appendices}

\end{document}